[1994/12/01]
\documentclass[12pt,reqno]{article}

\title{Multi-type TASEP in discrete time}
\author{James Martin and  Philipp Schmidt}
\date{February 18, 2010}
\usepackage{amsmath,amsthm,amsfonts}
\usepackage{t1enc}
\usepackage{bbm}
\usepackage{pstricks}
\usepackage{pst-node}
\usepackage{pst-plot}
\usepackage{pst-tree}
\usepackage{pst-grad}
\usepackage{pst-coil}
\usepackage{pst-text}
\usepackage{pst-3d}
\usepackage{pst-eps}
\usepackage[tiling]{pst-fill}
\usepackage{pstricks-add}
\usepackage{graphicx}
\usepackage[format=hang]{caption}
\usepackage{textcomp}

\usepackage[english]{babel}
\usepackage{nccmath}

\newtheorem{theorem}{Theorem}[section]
\newtheorem{cor}[theorem]{Corollary}
\newtheorem{lemma}[theorem]{Lemma}
\newtheorem{proposition}[theorem]{Proposition}

\theoremstyle{definition}

\theoremstyle{remark}
\newtheorem{remark}{Remark}[section]

\newcommand{\ZZ}{{\mathbb{Z}}}
\newcommand{\tT}{{\tilde{T}}}

\begin{document}
\maketitle 
\abstract 

The TASEP (\textit{totally asymmetric simple exclusion process}) is a
basic model for an one-dimensional interacting particle system with
non-reversible dynamics. Despite the simplicity of the model it shows
a very rich and interesting behaviour. In this paper we study some
aspects of the TASEP in discrete time and compare the results to the
recently obtained results for the TASEP in continuous time. 
In particular we focus on stationary distributions for
multi-type models, speeds of second-class particles,
collision probabilities and the ``speed process''. 
In discrete time, jump attempts may occur at different sites simultaneously, 
and the order in which these attempts are processed is important; 
we consider various natural update rules. 
\newline {\scshape Keywords:} TASEP, multi-type, second class particle, speed process.
\newline {\scshape AMS 2000 Mathematics Subject Classification:}
82C22, 60K35
\renewcommand{\sectionmark}[1]{}

\section{Introduction}

The TASEP in continuous time was introduced by Spitzer in 1970
(\cite{spitzer}) and can be described as follows. It is a Markov
process $\left( \eta_t \right)_{t \geq 0}$ on the state space $E =
\left\{ 0,1 \right\}^{\mathbb{Z}}$ where for $x \in \mathbb{Z}$ we
have that site $x$ is occupied with a particle at time $t$ iff
$\eta_t(x) =1$. Otherwise we say that site $x$ is empty at time $t$.
Starting from some initial configuration $\eta_0 \in E$,
\textit{updates} occur at each site as a Poisson process of rate 1,
independently; when an update occurs at site $x$, if there is a
particle at site $x$ and a hole to its right at site $x+1$, the
particle jumps from site $x$ to site $x+1$. If site $x$ is empty, or
if site $x+1$ is already occupied, the update has no effect.

In the model in discrete time, updates occur with some probability
$\beta \in \left( 0,1 \right)$ at each site at each time-step. Since
updates occur simultaneously, we now have to choose an order in which
to update the sites. We will consider sequential updates (from right
to left or from left to right) and sublattice parallel updates (even
sites first then odd sites).

For the model in continuous time there exists a vast amount of
literature. For an introduction and background to the topic see
Liggett's books \cite{liggett1} (pp. 361-417) and \cite{liggett2}
(pp. 209-316). However, in some physical models of interest it might
be more natural to use a discrete time scale. For example in traffic
models we can consider the reaction time of individuals as a smallest
time scale (Blythe and Evans \cite{blytheevans}, Chowdhury, Santen and
Schadschneider \cite{chowdhurysantenschadschneider} and Helbing
\cite{helbing}) and this suggests modelling traffic with a model in
discrete time. The ASEP (\textit{asymmetric simple exclusion process},
particles jump to the right at rate $p$ and to the left at rate $q < p$)
in discrete time was studied for example in Sch\"{u}tz \cite{schutz},
Hinrichsen \cite{hinrichsen}, Rajewsky, Santen, Schadschneider and
Schreckenberg \cite{rajewskysantenschadschneiderschreckenberg} and
Blythe and Evans \cite{blytheevans}. However, the behaviour of the
models in discrete time has not been analysed in as much depth as the
model in continuous time. The papers mentioned above are mainly
concerned with the model on a finite interval with open boundary
conditions and just one type of particles and analyse density profiles
and stationary distributions.

In this paper we derive further results for the TASEP in
discrete time that correspond to recently obtained results for the 
continuous-time model. These include 
stationary distributions for multi-type systems (e.g.\ \cite{ferrarimartin,
ferrarimartin2}),
laws of large numbers for the path of a second class particle 
and their connection to competition interfaces in competition
growth models (e.g.\ \cite{ferraripimentel, ferrarigoncalvesmartin}), 
and the \textit{TASEP speed process}
recently studied by Amir, Angel and Valk\'{o} \cite{amirangelvalko}.

We find that the multi-type invariant distributions for the models
with sequential updates are identical with those for the model 
in continuous time, and do not depend on the parameter $\beta$. 
This has the surprising consequence that 
various collision probabilities for different particles
in a multi-type processes started out of equilibrium, 
of the sort considered in \cite{ferrarigoncalvesmartin} and 
\cite{amirangelvalko}, are also independent of $\beta$ 
and coincide with the values for a continuous-time process. 
These probabilities correspond to survival probabilities
of clusters in the associated multi-type competition growth models. 
At the moment, the only argument we have for this property 
is indirect, using the fact that the set of invariant measures
is identical for all $\beta$; we do not know of a more 
direct argument based on local dynamics or couplings.

By contrast, in the case of sublattice-parallel updates,
the value of $\beta$ plays an important role in the 
set of stationary distributions. We extend the queue-based
construction of the multi-type stationary distributions 
from \cite{ferrarimartin, ferrarimartin2} by incorporating
queues whose arrival and service rates are different
at even and odd times.

The paper is organized as follows. In Section 2 we
will give a more formal definition of the model and introduce the
multi-type TASEP. The main results are described in Section 3,
including results concerning 
invariant measures and hydrodynamic limits for single-type models which
are required in order to state and understand the multi-type models
described above. The proofs or proof sketches for the
novel results are found in Section 4. In Section 5 we make some
brief remarks about a related discrete-time TASEP model 
with ``fully parallel updates''. 

\section{Model}
\subsection{Models in continuous and discrete time}
The TASEP in continuous time can be described by its generator $L$. For cylinder functions $f: E \rightarrow \mathbb{R}$ we have
$$ Lf(\eta) = \sum_{x \in \mathbb{Z}} \eta(x) \left(1 - \eta(x+1) \right) \left[ f \left( \eta^{x,x+1} \right) - f \left( \eta \right) \right] $$
with the configuration $\eta^{x,x+1}$ defined by
$$ \eta^{x,x+1}(y) =
\begin{cases}
	\eta(y) & y \notin \left\{ x, x+1 \right\} \\
	\eta(x+1) & y = x \\
	\eta(x) & y = x+1
\end{cases}
$$
Following ideas of Harris (1978) \cite{harris} we can use the
following graphical construction for the TASEP. Let $\{ \left( P^x_t
\right)_{t \geq 0} : x \in \mathbb{Z} \}$ be a family of independent
mean $1$ Poisson processes on a common probability space $\left(
  \Omega, \mathcal{A}, \mathbb{P} \right)$. For $x \in \mathbb{Z}$ the
process $P^x$ marks possible jumps from site $x$: If $P^x_{t} -
P^x_{t-} = 1$ and $\eta_{t-}(x)=1$ then the particle at $x$ tries to
jump one step to the right at time $t$. The jump is successful if the
adjacent site $x+1$ was unoccupied, i.e. $\eta_{t-}(x+1)=0$. Note that
for every $t>0$ with positive probability ($e^{-t}$) there was no jump
in the Poisson process $P^x$ up to time $t$. Since all the Poisson
processes are independent there will be infinitely many sites $x$ such
that there were no jumps in $P^x$. These sites separate $\mathbb{Z}$
into intervals of finite length. Since no particle can have crossed
the boundaries of these intervals, it is enough to 
be able to construct the process separately on each of these finite intervals.


We can use the same graphical construction to define the TASEP in
discrete time. All we have to do is replace the family of Poisson
processes with a family $\{ \left( B^x_n \right)_{n \geq 0} : x \in
\mathbb{Z} \}$ of independent Bernoulli processes with parameter $\beta
\in \left( 0,1 \right)$ and decide on an update rule for the sites. As
mentioned in the introduction we will mainly consider the following
three update rules:
\begin{itemize}
 \item Rule R1: Updates are processed in order from right to left.
 \item Rule R2: Updates are processed in order from left to right.
 \item Rule R3: All updates at even sites are processed before all updates at odd sites.
\end{itemize}
To highlight the difference between the three update rules we can look at the following example:
\begin{figure}[t]
\begin{center}
\begin{pspicture}(0,-0.5)(5,1)
	\psline{->}(0,0)(5,0)
	\psdots(1,0)(2,0)(3,0)(4,0)
	\cnode(1,0.5){0.3}{A1}
	\cnode(2,0.5){0.3}{A1}
	\uput{0.2}[270](1,0){*}
	\uput{0.2}[270](2,0){*}
	\uput{0.2}[270](3,0){*}
	\uput{0.5}[270](1,0){-1}
	\uput{0.5}[270](2,0){0}
	\uput{0.5}[270](3,0){1}
	\uput{0.5}[270](4,0){2}
\end{pspicture}
\end{center}
\setcaptionmargin{1cm}
\caption{Configuration at time $n$ and jump marks} \label{TDP}
\begin{center}
\begin{pspicture}(0,-0.5)(5,1)
	\psline[arrows=->](0,0)(5,0)
	\psdots(1,0)(2,0)(3,0)(4,0)
	\cnode(2,0.5){0.3}{A1}
	\cnode(3,0.5){0.3}{A1}
	\psarc{<-}(1.5,0.5){0.5}{40}{140}
	\psarc{<-}(2.5,0.5){0.5}{40}{140}
	\uput{0.5}[270](1,0){-1}
	\uput{0.5}[270](2,0){0}
	\uput{0.5}[270](3,0){1}
	\uput{0.5}[270](4,0){2}
\end{pspicture}
\end{center}
\setcaptionmargin{1cm}
\caption{Configuration at time $n+1$ if we apply R1} \label{TDP1}
\begin{center}
\begin{pspicture}(0,-0.5)(5,1)
	\psline[arrows=->](0,0)(5,0)
	\psdots(1,0)(2,0)(3,0)(4,0)
	\cnode(1,0.5){0.3}{A1}
	\cnode(4,0.5){0.3}{A1}
	\psarc{<-}(2.5,0.5){0.5}{40}{140}
	\psarc{<-}(3.5,0.5){0.5}{40}{140}
	\uput{0.5}[270](1,0){-1}
	\uput{0.5}[270](2,0){0}
	\uput{0.5}[270](3,0){1}
	\uput{0.5}[270](4,0){2}
\end{pspicture}
\end{center}
\setcaptionmargin{1cm}
\caption{Configuration at time $n+1$ if we apply R2} \label{TDP2}
\begin{center}
\begin{pspicture}(0,-0.5)(5,1)
	\psline[arrows=->](0,0)(5,0)
	\psdots(1,0)(2,0)(3,0)(4,0)
	\uput{0.5}[270](1,0){-1}
	\uput{0.5}[270](2,0){0}
	\uput{0.5}[270](3,0){1}
	\uput{0.5}[270](4,0){2}
	\cnode(2,0.5){0.3}{A1}
	\cnode(4,0.5){0.3}{A1}
	\psarc{<-}(2.5,0.5){0.5}{40}{140}
	\psarc{<-}(1.5,0.5){0.5}{40}{140}
	\psarc{<-}(3.5,0.5){0.5}{40}{140}
\end{pspicture}
\end{center}
\setcaptionmargin{1cm}
\caption{Configuration at time $n+1$ if we apply R3} \label{TDP3'}
\end{figure}
\\

Say we are at time $n$ in the configuration displayed in Figure \ref{TDP}, with particles at sites $-1$ and $0$ and holes at sites $1$ and $2$. There are jump attempts at the sites marked with a $*$. The resulting configurations under the three different update rules are as shown in Figures \ref{TDP1} - \ref{TDP3'}.

Note that in R2, a single particle may jump several times at the same time-step (but jumps are only possible onto sites that were already empty at the beginning of the time-step). In R1, several neighbouring particles may jump together at the same time-step. There is a natural symmetry between systems R1 and R2 - one is transformed into the other by exchanging left and right and exchanging the roles of particle and hole. For the last example (R3) the parity of the sites is important.

In connection with the \textit{speed process} we will also mention the model with odd/even updates (R4). Again this can be obtained from R3 by a simple transformation.

As seen above, each of these models shows a
slightly different behaviour, but if we rescale time by a factor
$\beta^{-1}$ and let $\beta \rightarrow 0$ then they converge to the
model in continuous time. In this sense the model in discrete time is
more general than the model in continuous time (which in the following
we will denote by R0) since we can recover the model in continuous
time from the model in discrete time. In discrete time we can also
consider the model with (fully) parallel updates where all sites are
updated simultaneously. However, many of the methods developed for
the model in continuous time that work in the models R1-R3 fail in
this case. We will mention some questions connected to this model in
Section 5.

\subsection{Percolation representations}\label{percrep}
Both in continuous and in discrete time, one important feature of the
TASEP is its connection to \textit{last-passage percolation} and the
\textit{corner growth model}. Here we consider a special case which
corresponds to a particular initial condition of the TASEP, 
in which, at time 0, all non-positive sites $x\leq0$ contain a particle 
and all positive sites $x>0$ are empty. We label the particles from 
right to left, so that for $i\geq 1$, particle $i$ starts at 
site $-i+1$ at time 0 (and always remains to the right of particle $i+1$).


For $n,k\geq1$, let $T(n,k)$ be the time that particle $k$ jumps to its right for the 
$n$th time. Then it is well-known that the variables $T(n,k)$ satisfy the recursions
\begin{equation}\label{Tnkrecursion}
T(n,k) = \max \left\{ T(n-1,k) , T(n,k-1) \right\} + v(n,k) \qquad n,k \geq 1
\end{equation}
with boundary conditions $T(0,k)=T(n,0)=0$ for all $n,k$, 
where $v(n,k)$ are i.i.d.\ exponential random variables with mean 1.
The interpretation is that before particle $k$ can make its $n$th jump,
both particle $k$ must have made its $(n-1)$st jump, and particle $k-1$ 
must have made its $n$th jump. Once these two events have happened, 
an amount of time which is exponentially distributed with rate 1 passes
before particle $k$ makes its $n$th jump; this is the random variable $v(n,k)$. 

The random variables $T(n,k)$ have an interpretation in terms of last-passage
percolation times. 
For an increasing path $\pi$ from $z \in \ZZ_+^2$ to
$z' \in \ZZ_+^2$, i.e. a path with increments in
$\{ (0,1) , (1,0) \}$, define the weight of $\pi$ by
$$ S \left( \pi \right) = \sum_{z'' \in \pi} v(z''). $$
Write $\Pi(z,z')$ for the set of all increasing paths from $z$ to $z'$; then
\begin{equation}
R(z,z') = \max_{\pi \in \Pi(z,z')} S(\pi) \label{R}
\end{equation}
is the weight of the heaviest path from $z$ to $z'$. 
Then, via the recursions (\ref{Tnkrecursion}),
it is easy to see that $T(n,k)=R((1,1),(n,k))$.
In this setting we may interpret the random variable $v(n,k)$ as a weight
at the lattice point $(n,k)$.

We turn to the discrete-time case. Now let $w(n,k)$ be i.i.d.\ random
variables whose distribution is geometric with parameter $\beta \in (0,1)$ (by
which we mean that $\mathbb{P} \left[ w(z)=k \right] = (1-\beta)^{k}\beta$ for
$k=0,1,2,\dots$).
We define passage-times $\tT(n,k)$ 
analogous to $T(n,k)$ above by the recursions
\begin{equation*}
\tT(n,k) = \max \left\{ \tT(n-1,k) , \tT(n,k-1) \right\} + w(n,k) \qquad n,k \geq 1.
\end{equation*}
We will describe three variants on these recursions, which pertain to 
the different update rules R1, R2 and R3. As above, $w(n,k)$ 
will correspond to the delay before particle $k$ makes its $n$th jump, 
once it is free to do so. 
For $i=1,2,3$, let
$T^{(i)}(n,k)$ be the $n$th jump of particle $k$ under update rule R$i$ with boundary conditions $T^{(i)}(0,k) = T^{(i)}(n,0) = -1$ for all $n,k$.

\bigskip

\noindent
\textbf{Rule R1 (updates from right to left)}
  \begin{itemize}
	\item Recursions:
		\begin{align}
		T^{(1)}(n,k)
		&= \max \left\{ T^{(1)}(n-1,k) +1 , T^{(1)}(n,k-1) \right\} + w(n,k) \notag \\
		&= \tT(n,k) + n - 1 \label{T1}
		\end{align}
	\item In accordance with the updates from right to left, particles $k$ and $k-1$ can make their $n$th jumps
	at the same time-step, but two jumps by the same particle must be separated by at least one time-step. 
	\item This corresponds to a percolation model in which as well as weights $w(n,k)$ at the vertices
	$(n,k)\in\ZZ_+^2$, we have weights of size 1 on each horizontal edge between $(n-1,k)$ and $(n,k)$.
  \end{itemize}
\textbf{Rule R2 (updates from left to right)}
  \begin{itemize}
	\item Recursions:
		\begin{align}
		T^{(2)}(n,k)
		&= \max \left\{ T^{(2)}(n-1,k) , T^{(2)}(n,k-1) +1 \right\} + w(n,k) \notag \\
		&= \tT(n,k) + k - 1 \label{T2}
		\end{align}
	\item With updates from left to right, a particle may make several jumps at the same time-step,
	but at least one time-step must separate the $n$th jump of particles $k-1$ and $k$. 
	\item In the corresponding percolation model, the weights of size 1 are now on the vertical edges of the lattice.
\end{itemize}
\textbf{Rule R3 (even updates then odd updates)}
  \begin{itemize}
	\item Recursions:
		\begin{align}
	 	T^{(3)}(n,k)
		&= \begin{cases}
			\max \left\{ T^{(3)}(n-1,k) + 1, T^{(3)}(n,k-1) + 1 \right\} + w(n,k) \\
			\qquad \qquad \qquad \qquad \qquad \qquad \qquad \qquad \qquad \qquad \qquad n+k \text{ even} \\
			\max \left\{ T^{(3)}(n-1,k), T^{(3)}(n,k-1) \right\} + w(n,k) \\
			\qquad \qquad \qquad \qquad \qquad \qquad \qquad \qquad \qquad \qquad \qquad n+k \text{ odd}
		   \end{cases} \notag \\
		&= \begin{cases}
			\tT(n,k) + \frac{n+k-2}{2} & n+k \text{ even} \label{T3} \\
			\tT(n,k) + \frac{n+k-3}{2} & n+k \text{ odd}
		   \end{cases}
		\end{align}
	\item Now the edge weights of size 1 are added to all edges with an upper/right point $(n,k)$
such that $n+k$ is even.
\end{itemize}
For the model in continuous time we have, for $x > 0$,
\begin{equation}
\lim_{n \rightarrow \infty} \frac{T([xn],n)}{n} = \left( \sqrt{x} + 1 \right)^2 \qquad \text{a.s.} \label{LPPexp}
\end{equation}
This was essentially first shown in \cite{rost}.
Replacing the exponential weights by geometric weights gives
\begin{equation}
\lim_{n \rightarrow \infty} \frac{\tT([xn],n)}{n} = \frac{(1-\beta)x + 2 \sqrt{(1-\beta)x} + (1-\beta)}{\beta} \qquad 
\text{a.s.}; \label{LPPgeom}
\end{equation}
see for example \cite{oconnell}.
Using (\ref{T1})-(\ref{T3}), this can easily be used to give similar laws of large numbers for
$T^{(i)}([xn],n)$, $i=1,2,3$.


We may also view the system as a growth model.
For the continuous-time case,
\begin{equation}\label{Gdef}
G_t = \left\{ (x,y) \in \ZZ_+^2 : T(x,y) \leq t \right\}
\end{equation}
be the set of vertices whose passage-time is less than $t$. This gives
a cluster in $\ZZ_+^2$ which grows over time; there is a 1-1 correspondence 
relating $G_t$ to the configuration of the TASEP at time $t$;
the length of the row at height $k\in\ZZ_+$ is the number of jumps 
particle $k$ has made in the TASEP. 
In a similar way we can define $G^{(1)}(t)$, $G^{(2)}(t)$ and $G^{(3)}(t)$ 
by replacing $T$ in (\ref{Gdef}) by $T^{(1)}$, $T^{(2)}$ or $T^{(3)}$ respectively.


\subsection{Multi-type models}
In the \textit{multi-type TASEP} each particle belongs to a class $y \in
\mathbb{Z}$ (or more generally $y \in \mathbb{R}$). All particles can
still jump into unoccupied sites. When a particle of class $k$ tries
to jump into a site which is occupied by a particle of class $j$ two
things can happen: If $k \geq j$ the jump is suppressed and if $k < j$
then the two particles swap. This means that the lower the class of a
particle the higher is its priority. 

An $N$-type TASEP (containing $N$ classes of particles and holes) 
can be regarded as a coupling of 
$N$ ordered single-type TASEPs. If $\eta^1_0$, $\ldots$, $\eta^N_0$ are
$N$ TASEP configurations such that $\eta^1_0(x) \leq \ldots \leq
\eta^N_0(x)$ for all $x \in \mathbb{Z}$, we can use the same Poisson or
Bernoulli processes (this is called \textit{basic coupling}) to get a
joint realization of the TASEPs $\eta^1$, $\ldots$, $\eta^N$.

The basic coupling preserves the ordering between the processes
(since the updates are processed one by one, this is true
for the discrete-time models just as in the continuous-time case). 
Thus we can define a multi-type 
process $\xi$ by
$$ \xi_t(x) = N + 1 - \sum_{k=1}^N \eta^k_t(x). $$
We write $\xi_t=R\eta_t$.
Particles of class $k$ occur at sites $x$ where $\xi(x)=k$.
For $k>1$,
these sites represent discrepancies between the processes
$\eta^{k-1}$ and $\eta^k$. We may regard particles of type $N+1$ as holes.
Then $\xi$ behaves
like a multi-type TASEP with $N$ classes of particles and holes.
See for example \cite{ferrarimartin} for further details.

\section{Results}
We will divide this section into three subsections: The first deals
with invariant measures for single- and multi-type models, the second
with hydrodynamic limits and the third with multi-type models out of
equilibrium.

\subsection{Invariant measures}

\begin{proposition} \label{invariant} For the TASEP in continuous time as well as the discrete time TASEPs R1 and R2, the Bernoulli product measures $\nu_{\rho}$ with marginals $\rho \in \left[ 0,1 \right]$ are the only translation invariant stationary ergodic measures with constant marginals. For the TASEP R3, the Bernoulli product measures $\mu_{\rho}$ with marginals $\rho \in \left[0,1 \right]$ on even sites and marginals $\frac{\rho (1 - \beta)}{1 - \rho \beta}$ on odd sites are the only stationary ergodic measures with marginals that are translation invariant under even shifts.
\end{proposition}

\begin{remark}
Interestingly, the marginals of the invariant Bernoulli product measures for the models R1 and R2 do not depend on the model parameter $\beta$, and coincide with the invariant measures for the model in continuous time. In the model R3 however, the densities at even and odd sites differ (with a specific relation between them) and the measure depends on the parameter $\beta$.
 \end{remark}

\begin{proof}
For references see for example Liggett \cite{liggett1} for R0, Blythe and Evans \cite{blytheevans} for R1,R2 and Rajewsky, Santen, Schadschneider and Schreckenberg \cite{rajewskysantenschadschneiderschreckenberg} for R3. The uniqueness statements can be proved following the approach of Mountford and Prabhakar \cite{mountfordprabhakar}.
\end{proof}

We now turn to the construction of invariant measures for systems
with more than one class of particles. We use the
construction based on a system of queues in tandem developed in
\cite{ferrarimartin}, and begin by recalling notation from that paper.

Given two processes $\alpha_1$ and $\alpha_2$, 
taking values in $\{0,1\}^\ZZ$ and representing
the arrival and service processes of a queue respectively,
let $D(\alpha_1,\alpha_2)$ 
be the process of departures from the queue. 
Now define $D^{(1)}(\alpha)=\alpha$, $D^{(2)}(\alpha_1,\alpha_2)=
D(\alpha_1,\alpha_2)$, and recursively
$D^{(n)} \left( \alpha_1, \ldots , \alpha_n
\right) = D \left( D^{(n-1)} \left( \alpha_1 , \ldots , \alpha_{n-1}
  \right) , \alpha_n \right)$ for $n>2$. 
(The process $D^{(n)}$ can be seen as the departure process from a 
system of $n-1$ queues in tandem). 
Now for $\alpha=(\alpha_1,\dots,\alpha_n)$ we can define a system of $n$ ordered
single-type TASEP configurations, denoted $T\alpha = \eta = \left( \eta^1 ,
  \ldots , \eta^n \right)$ by $ \eta^k = D^{(n-k+1)} \left( \alpha_k ,
  \ldots , \alpha_n \right) $. Then the corresponding
multi-type configuration $\xi = \xi^{(1,\ldots,n)}$ is given by
$\xi=R\eta=RT\alpha$, with $\xi(x) = n + 1 - \sum_{k=1}^n \eta^k(x) $ (as in
the last paragraph of Section 2). See Remark \ref{queueexplanation} below
for further explanation of the construction.

We can now state the main result. 

To state this result, we work with systems with jumps 
\textit{from right to left}. To return to the systems defined before,
one simply takes the space-reversal ($\tilde\eta_t(x)=\eta_t(-x)$). 
Note that time in the queueing system corresponds to space in the particle system. 

\begin{theorem} \label{invariantmulti} If $\alpha = \left( \alpha_1 ,
    \ldots , \alpha_n \right)$ has distribution $\nu = \nu_{\rho_1}
  \times \ldots \times \nu_{\rho_n}$ ($\mu = \mu_{\rho_1} \times
  \ldots \times \mu_{\rho_n}$ respectively for model R3) with $\rho_1
  < \ldots < \rho_n$, then the law of $T \alpha = \eta$ is invariant
  for the coupled multi-line TASEPs R0,R1 and R2 (R3 respectively) and
  the law of $RT\alpha = R \eta = \xi$ is invariant for the multi-type
  TASEPs R0, R1 and R2 (R3 respectively) with jumps from right to
  left. These are the unique stationary translation invariant
  (invariant under even shifts respectively) ergodic measures with
  density $\rho_1$ of first class particles (density $\rho_1$ of first
  class particles on even sites), density $\rho_2 - \rho_1$ of second
  class particles (density $\rho_2 - \rho_1$ of second class particles
  on even sites), etc.
\end{theorem}

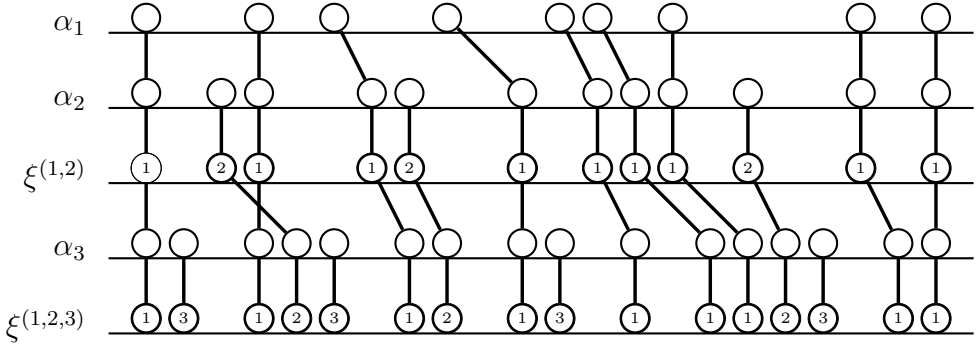
\begin{figure}[t]
\begin{center}
\begin{pspicture}(-3,-3.5)(8,2)
	\psline(-2.5,1)(9,1)
	\psline(-2.5,0)(9,0)
	\psline(-2.5,-1)(9,-1)
	\psline(-2.5,-2)(9,-2)
	\psline(-2.5,-3)(9,-3)
	\cnode(-2,1.2){0.2}{A1}
	\cnode(-0.5,1.2){0.2}{A2}
	\cnode(0.5,1.2){0.2}{A3}
	\cnode(2,1.2){0.2}{A4}
	\cnode(3.5,1.2){0.2}{A5}
	\cnode(4,1.2){0.2}{A6}
	\cnode(5,1.2){0.2}{A7}
	\cnode(7.5,1.2){0.2}{A8}
	\cnode(8.5,1.2){0.2}{A9}
	\cnode(-2,0.2){0.2}{B1}
	\cnode(-1,0.2){0.2}{B2}
	\cnode(-0.5,0.2){0.2}{B3}
	\cnode(1,0.2){0.2}{B4}
	\cnode(1.5,0.2){0.2}{B5}
	\cnode(3,0.2){0.2}{B6}
	\cnode(4,0.2){0.2}{B7}
	\cnode(4.5,0.2){0.2}{B8}
	\cnode(5,0.2){0.2}{B9}
	\cnode(6,0.2){0.2}{B10}
	\cnode(7.5,0.2){0.2}{B11}
	\cnode(8.5,0.2){0.2}{B12}
	\cput[linewidth=0.1pt](-2,-0.8){\tiny 1}
	\cnode[linewidth=0pt](-2,-0.8){0.2}{C1}
	\cput[linewidth=0.1pt](-1,-0.8){\tiny 2}
	\cnode(-1,-0.8){0.2}{C2}
	\cput[linewidth=0.1pt](-0.5,-0.8){\tiny 1}
	\cnode(-0.5,-0.8){0.2}{C3}
	\cput[linewidth=0.1pt](1,-0.8){\tiny 1}
	\cnode(1,-0.8){0.2}{C4}
	\cput[linewidth=0.1pt](1.5,-0.8){\tiny 2}
	\cnode(1.5,-0.8){0.2}{C5}
	\cput[linewidth=0.1pt](3,-0.8){\tiny 1}
	\cnode(3,-0.8){0.2}{C6}
	\cput[linewidth=0.1pt](4,-0.8){\tiny 1}
	\cnode(4,-0.8){0.2}{C7}
	\cput[linewidth=0.1pt](4.5,-0.8){\tiny 1}
	\cnode(4.5,-0.8){0.2}{C8}
	\cput[linewidth=0.1pt](5,-0.8){\tiny 1}
	\cnode(5,-0.8){0.2}{C9}
	\cput[linewidth=0.1pt](6,-0.8){\tiny 2}
	\cnode(6,-0.8){0.2}{C10}
	\cput[linewidth=0.1pt](7.5,-0.8){\tiny 1}
	\cnode(7.5,-0.8){0.2}{C11}
	\cput[linewidth=0.1pt](8.5,-0.8){\tiny 1}
	\cnode(8.5,-0.8){0.2}{C12}
	\cnode(-2,-1.8){0.2}{D1}
	\cnode(-1.5,-1.8){0.2}{D2}
	\cnode(-0.5,-1.8){0.2}{D3}
	\cnode(0,-1.8){0.2}{D4}
	\cnode(0.5,-1.8){0.2}{D5}
	\cnode(1.5,-1.8){0.2}{D6}
	\cnode(2,-1.8){0.2}{D7}
	\cnode(3,-1.8){0.2}{D8}
	\cnode(3.5,-1.8){0.2}{D9}
	\cnode(4.5,-1.8){0.2}{D10}
	\cnode(5.5,-1.8){0.2}{D11}
	\cnode(6,-1.8){0.2}{D12}
	\cnode(6.5,-1.8){0.2}{D13}
	\cnode(7,-1.8){0.2}{D14}
	\cnode(8,-1.8){0.2}{D15}
	\cnode(8.5,-1.8){0.2}{D16}
	\cnode(-2,-2.8){0.2}{E1}
	\cput[linewidth=0.1pt](-2,-2.8){\tiny 1}
	\cnode(-1.5,-2.8){0.2}{E2}
	\cput[linewidth=0.1pt](-1.5,-2.8){\tiny 3}
	\cnode(-0.5,-2.8){0.2}{E3}
	\cput[linewidth=0.1pt](-0.5,-2.8){\tiny 1}
	\cnode(0,-2.8){0.2}{E4}
	\cput[linewidth=0.1pt](0,-2.8){\tiny 2}
	\cnode(0.5,-2.8){0.2}{E5}
	\cput[linewidth=0.1pt](0.5,-2.8){\tiny 3}
	\cnode(1.5,-2.8){0.2}{E6}
	\cput[linewidth=0.1pt](1.5,-2.8){\tiny 1}
	\cnode(2,-2.8){0.2}{E7}
	\cput[linewidth=0.1pt](2,-2.8){\tiny 2}
	\cnode(3,-2.8){0.2}{E8}
	\cput[linewidth=0.1pt](3,-2.8){\tiny 1}
	\cnode(3.5,-2.8){0.2}{E9}
	\cput[linewidth=0.1pt](3.5,-2.8){\tiny 3}
	\cnode(4.5,-2.8){0.2}{E10}
	\cput[linewidth=0.1pt](4.5,-2.8){\tiny 1}
	\cnode(5.5,-2.8){0.2}{E11}
	\cput[linewidth=0.1pt](5.5,-2.8){\tiny 1}
	\cnode(6,-2.8){0.2}{E12}
	\cput[linewidth=0.1pt](6,-2.8){\tiny 1}
	\cnode(6.5,-2.8){0.2}{E13}
	\cput[linewidth=0.1pt](6.5,-2.8){\tiny 2}
	\cnode(7,-2.8){0.2}{E14}
	\cput[linewidth=0.1pt](7,-2.8){\tiny 3}
	\cnode(8,-2.8){0.2}{E15}
	\cput[linewidth=0.1pt](8,-2.8){\tiny 1}
	\cnode(8.5,-2.8){0.2}{E16}
	\cput[linewidth=0.1pt](8.5,-2.8){\tiny 1}
	\ncline[linewidth=1.3pt]{-}{A1}{B1}
	\ncline[linewidth=1.3pt]{-}{A2}{B3}
	\ncline[linewidth=1.3pt]{-}{A3}{B4}
	\ncline[linewidth=1.3pt]{-}{A4}{B6}
	\ncline[linewidth=1.3pt]{-}{A5}{B7}
	\ncline[linewidth=1.3pt]{-}{A6}{B8}
	\ncline[linewidth=1.3pt]{-}{A7}{B9}
	\ncline[linewidth=1.3pt]{-}{A8}{B11}
	\ncline[linewidth=1.3pt]{-}{A9}{B12}
	\ncline[linewidth=1.3pt]{-}{B1}{C1}
	\ncline[linewidth=1.3pt]{-}{B2}{C2}
	\ncline[linewidth=1.3pt]{-}{B3}{C3}
	\ncline[linewidth=1.3pt]{-}{B4}{C4}
	\ncline[linewidth=1.3pt]{-}{B5}{C5}
	\ncline[linewidth=1.3pt]{-}{B6}{C6}
	\ncline[linewidth=1.3pt]{-}{B7}{C7}
	\ncline[linewidth=1.3pt]{-}{B8}{C8}
	\ncline[linewidth=1.3pt]{-}{B9}{C9}
	\ncline[linewidth=1.3pt]{-}{B10}{C10}
	\ncline[linewidth=1.3pt]{-}{B11}{C11}
	\ncline[linewidth=1.3pt]{-}{B12}{C12}
	\ncline[linewidth=1.3pt]{-}{C1}{D1}
	\ncline[linewidth=1.3pt]{-}{C2}{D4}
	\ncline[linewidth=1.3pt]{-}{C3}{D3}
	\ncline[linewidth=1.3pt]{-}{C4}{D6}
	\ncline[linewidth=1.3pt]{-}{C5}{D7}
	\ncline[linewidth=1.3pt]{-}{C6}{D8}
	\ncline[linewidth=1.3pt]{-}{C7}{D10}
	\ncline[linewidth=1.3pt]{-}{C8}{D11}
	\ncline[linewidth=1.3pt]{-}{C9}{D12}
	\ncline[linewidth=1.3pt]{-}{C10}{D13}
	\ncline[linewidth=1.3pt]{-}{C11}{D15}
	\ncline[linewidth=1.3pt]{-}{C12}{D16}
	\ncline[linewidth=1.3pt]{-}{D1}{E1}
	\ncline[linewidth=1.3pt]{-}{D2}{E2}
	\ncline[linewidth=1.3pt]{-}{D3}{E3}
	\ncline[linewidth=1.3pt]{-}{D4}{E4}
	\ncline[linewidth=1.3pt]{-}{D5}{E5}
	\ncline[linewidth=1.3pt]{-}{D6}{E6}
	\ncline[linewidth=1.3pt]{-}{D7}{E7}
	\ncline[linewidth=1.3pt]{-}{D8}{E8}
	\ncline[linewidth=1.3pt]{-}{D9}{E9}
	\ncline[linewidth=1.3pt]{-}{D10}{E10}
	\ncline[linewidth=1.3pt]{-}{D11}{E11}
	\ncline[linewidth=1.3pt]{-}{D12}{E12}
	\ncline[linewidth=1.3pt]{-}{D13}{E13}
	\ncline[linewidth=1.3pt]{-}{D14}{E14}
	\ncline[linewidth=1.3pt]{-}{D15}{E15}
	\ncline[linewidth=1.3pt]{-}{D16}{E16}
	\uput{0.3}[180](-2.5,1.1){$\alpha_1$}
	\uput{0.3}[180](-2.5,0.1){$\alpha_2$}
	\uput{0.3}[180](-2.5,-1.9){$\alpha_3$}
	\uput{0.3}[180](-2.5,-0.9){$\xi^{(1,2)}$}
	\uput{0.3}[180](-2.5,-2.9){$\xi^{(1,2,3)}$}
\end{pspicture}
\setcaptionmargin{1cm}
\caption{Queues in tandem and multi-type configurations $\xi^{(1,2)}$ and $\xi^{(1,2,3)}$} \label{fig1}
\end{center}
\end{figure}

\begin{remark}\label{queueexplanation}
  The mechanism to construct an invariant distribution as described
  above can be depicted in the following way: Take $\alpha_1$ as the
  arrival process and $\alpha_2$ as the service process of a
  queue. Using $\alpha_1$ and $\alpha_2$ we can construct a process
  consisting of the departures from this queue (first class
  particles), unused services (second class particles) and
  times when no service was offered (holes). We then use this process as the
  arrival process for a queue with service process $\alpha_3$ where
  first class particles have priority over second class particles: If
  there is a service and a first and a second class particle are
  waiting in the queue then the first class particles gets served
  first. In this way we get a resulting process consisting of
  departures of first class particles (first class particles),
  departures of second class particles (second class particles),
  unused services (third class particles) and holes. Now we can feed
  this process into a queue with service process $\alpha_4$ and so
  on. If $\alpha = \left( \alpha_1 , \ldots , \alpha_n \right)$ has
  distribution $\nu = \nu_{\rho_1} \times \ldots \times \nu_{\rho_n}$
  ($\mu = \mu_{\rho_1} \times \ldots \times \mu_{\rho_n}$
  respectively) then the distribution of the resulting multi-type
  configuration is invariant for the multi-type TASEP. See Figure
  \ref{fig1} for an illustration. Note that for models R0, R1, R2,
the queues involved are simply $M/M/1$ queues in discrete-time;
the same is almost true for R3, 
except that we have different arrival and service rates at 
odd and even times.
\end{remark}

\begin{remark}
We observe again that the invariant measures for the multi-type TASEPs R1 and R2 are the same as the invariant measures for the multi-type TASEP in continuous time and that they do not depend on $\beta$. Since the invariant measures for the single-type TASEP R3 depend on $\beta$ the same is true for the invariant measures for the multi-type TASEP R3.
\end{remark}


\subsection{Hydrodynamic limits}
We now move to considering systems out of equilibrium. 
We consider the particular initial configuration given by 
$$\eta_0(x) =
\begin{cases}
	1 & x \leq 0 \\
	0 & x \geq 1
\end{cases}
$$
This ``step'' initial condition corresponds to the 
corner growth model and to the particular initial conditions
for the percolation models described in Section \ref{percrep}.
We define the following functions $f_0$,
$f_1$, $f_2$ and $f_3$, which will describe the 
evolving density profile for 
the continuous-time TASEP and for 
the TASEPs R1-R3 in discrete time:
\begin{align*}
&f_0(u) = \mathbbm{1}_{ \left. \left( -\infty , -1 \right. \right]}(u) + \frac{1}{2}(1-u) \cdot \mathbbm{1}_{\left[ -1 , 1 \right]}(u) \\
&f_1(u) = \mathbbm{1}_{ \left. \left( -\infty , -\frac{\beta}{1-\beta} \right. \right]}(u) + \frac{1}{\beta} \left( 1 - \sqrt{\frac{1-\beta}{1-u}} \right) \cdot \mathbbm{1}_{\left[ -\frac{\beta}{1-\beta} , \beta \right]}(u) \\
&f_2(u) = \mathbbm{1}_{ \left. \left( -\infty , - \beta \right. \right]}(u) + \left( 1 - \frac{1}{\beta} \left( 1 - \sqrt{\frac{1-\beta}{1+u}} \right) \right) \cdot \mathbbm{1}_{\left[ - \beta , \frac{\beta}{1-\beta} \right]}(u) = 1 - f_1(-u) \\
&f_3(u) = \mathbbm{1}_{ \left. \left( -\infty , - \frac{2 \beta}{2 - \beta} \right. \right]}(u) + \left( \frac{1}{2} - \frac{u}{\beta} \sqrt{ \frac{1 - \beta}{4 - u^2}} \right) \cdot \mathbbm{1}_{\left[ - \frac{2 \beta}{2 - \beta} , \frac{2 \beta}{2 - \beta} \right]}(u) \\
\text{We }
&\text{let } a_3 \text{ be defined by } f_3(u) = \frac{1}{2} \left( a_3(u) + \frac{a_3(u) (1-\beta)}{1-a_3(u)\beta} \right) \text{, so} \qquad \qquad \qquad \qquad \qquad \qquad \qquad \\
&a_3(u) = \mathbbm{1}_{ \left. \left( -\infty , - \frac{2 \beta}{2 - \beta} \right. \right]}(u) + \frac{1}{\beta} \left( 1 - \left( 2 + u \right) \sqrt{\frac{1-\beta}{4-u^2}} \right) \cdot \mathbbm{1}_{\left[ - \frac{2 \beta}{2 - \beta} , \frac{2 \beta}{2 - \beta} \right]}(u)
\end{align*}
For $i=0,1,2,3$ let $\tau_i(k,t)$ ($t \in \mathbb{R}_+$ or $t \in \mathbb{N}$ respectively) be the distribution of $(\eta_{t}(k+l), l \in \mathbb{Z})$ in the corresponding model. We have the following result for the TASEP in continuous time and the discrete TASEPs R1-R3.

\begin{theorem} \label{hydro} For any $u \in \mathbb{R}$ and $i=0,1,2$ the
  measure $\tau_i \left( \left[ ut \right] , t \right)$ converges
  weakly to the Bernoulli product measure with marginals $f_i(u)$ and
  $\tau_3 \left( \left[ ut \right] , t \right)$ converges weakly to
  the Bernoulli product measure with marginals $a_3(u)$ on even sites
  and $\frac{a_3(u) (1-\beta)}{1-a_3(u)\beta}$ on odd
  sites. In particular we have that for any $u \in \mathbb{R}$ the limit
  $\lim_{t \rightarrow \infty} \mathbb{E} \left[ \eta_t(k) \right]$
  exists and is equal to $f_i(u)$, $i = 0,1,2$ depending on which
  model we are considering, whenever $\frac{k}{t}$ tends to $u$, and
  $\lim_{t \rightarrow \infty} \mathbb{E} \left[
    \eta_t(2[\frac{k}{2}]) \right]$ $= a_3(u)$ and $\lim_{t
    \rightarrow \infty} \mathbb{E} \left[ \eta_t(2[\frac{k}{2}] + 1)
  \right]= \frac{a_3(u) (1-\beta)}{1-a_3(u)\beta}$ in the model
  R3. Furthermore, for $i=0,1,2,3$, the quantities $\frac{1}{t}
  \sum_{ut < k < vt} \eta_t(k)$ converge a.s. to the constant value
  $\int_{u}^{v} f_i(w) dw$, for $u < v$.
\end{theorem}

The first part of the theorem states convergence to local equilibrium:
suitably rescaled the models converge locally to the unique invariant
measures from Theorem \ref{invariant}. This implies the other
statements of the theorem. However, in the models R0, R1 and R2 we can
prove the second part without proving convergence to local equilibrium
first, while in the model R3 our proof for the second part requires
convergence to local equilibrium. The statements for the model in
continuous time were proved for the first time by Rost
\cite{rost}. O'Connell \cite{oconnell} used the connection between the
TASEP and last-passage percolation to prove an equivalent result about
the asymptotic shape of the corner growth model (as defined in Section
2). The parts of Theorem \ref{hydro} concerning the models in discrete
time can be proved using exactly the same methods as Rost \cite{rost}
and O'Connell \cite{oconnell}. In Section 4 we will outline the proof
for the model R3.

\begin{figure}[t]
\begin{center}
\includegraphics[width=6.5cm,height=13cm, angle=270]{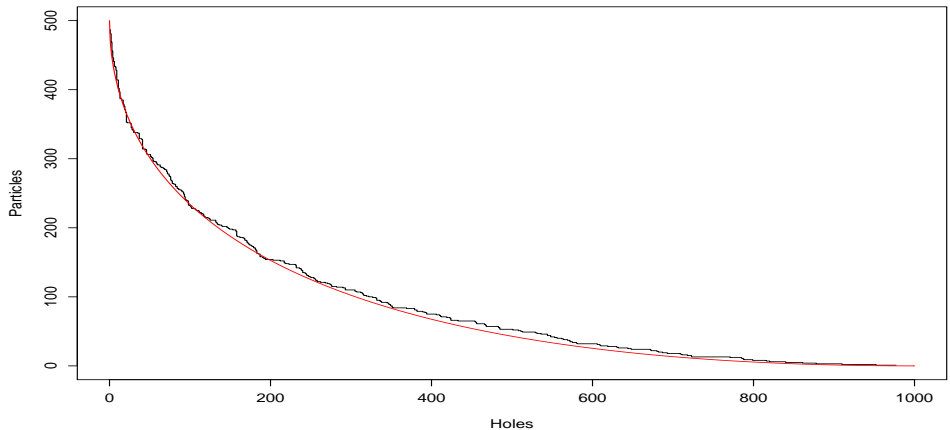}
\setcaptionmargin{1cm}
\caption{Black: Simulation of the corner growth model (R2) for $\beta=0.5$ and time $n=1000$; Red: Limiting shape given by a rescaled version of $g_2$} \label{simu1}
\end{center}
\end{figure}

\begin{remark} \label{shape}
From the convergence to the density profiles $f_0$, $f_1$,$f_2$ and $f_3$ we can easily deduce a shape theorem for the corner growth model defined in Section 2. The asymptotic shape in the models R1, R2 and R3 (after rescaling by $t$) are for example given by the functions
\begin{equation} \label{g1}
g_1(x) = \frac{1}{1-\beta} \left( \sqrt{\beta} - \sqrt{x} \right)^2
\end{equation}
for $x \in \left[ 0, \beta \right]$,
\begin{equation} \label{g2}
g_2(x) = \left( \sqrt{\beta} - \sqrt{(1-\beta)x} \right)^2
\end{equation}
for $x \in [ 0, \frac{\beta}{1-\beta} ]$, see Figure \ref{simu1} (simulation with $\beta=0.5$ up to time $n=1000$), and
\begin{equation} \label{g3}
g_3(x) = \frac{1}{(2-\beta)^2} \left( \sqrt{4x(1-\beta)} - \sqrt{4\beta - \beta^2x - 2\beta^2} \right)^2
\end{equation}
for $x \in \left[0 , \frac{2\beta}{2-\beta} \right]$.
\end{remark}


\bigskip

\begin{remark}
  We may rescale time as well as space in Theorem \ref{hydro},
  and look at the limit $f(u,t)=\lim_{N \rightarrow \infty} \mathbb{E}
  \left[ \eta_{Nt} (k) \right]$ for $\lim_{N \rightarrow \infty}
  \frac{k}{N} = u$. For the continuous-time model,
this density profile is governed by Burgers' equation, 
\begin{equation} \label{burger}
  \frac{\partial f}{\partial t} + \frac{\partial(f(1-f))}{\partial u} = 0; 
\end{equation}
the solution, with initial condition
$f(u,0) = \mathbbm{1}_{\left( \left. -\infty , 0 \right] \right.}(u),$
is 
$$ f(u,t) = \mathbbm{1}_{ \left. \left( -\infty , -t \right. \right]}(u) + \frac{t-u}{2t} \cdot \mathbbm{1}_{\left[ -t , t \right]}(u) $$
The differential equation also governs the evolution of the density profile
for more general initial configurations than the ``step'' initial
condition. We can get equations analogous to 
(\ref{burger}) for the
models in discrete time. For example, for model R1,
$$ 
f_1(u,t) = 
\mathbbm{1}_{ \left. \left( -\infty , -\frac{\beta t}{1-\beta} \right. \right]}(u) 
+ \frac{1}{\beta} \left( 1 - \sqrt{\frac{t(1-\beta)}{t-u}} \right) 
\cdot \mathbbm{1}_{\left[ -\frac{\beta t}{1-\beta} , \beta t \right]}(u)
$$
solves
\begin{equation} \label{burgerR1}
\begin{cases}
	\frac{\partial f_1}{\partial t} 
+ \frac{\partial}{\partial u}
\frac{\beta f_1(1-f_1)}{1-\beta f_1}
= 0 \\
	f_1(u,0) = \mathbbm{1}_{\left( \left. -\infty , 0 \right] \right.}(u)
\end{cases}
\end{equation}
Here $\frac{\beta f_1(1-f_1)}{1-\beta f_1} 
=\sum_{n=1}^{\infty}
\beta^n f_1^n(1-f_1)$ 
is the probability that a particle jumps from a
given site to its neighbour in a model in equilibrium with marginal
density $f_1$. 
\end{remark}

\subsection{Multi-type models out of equilibrium}\label{outof}

In this section we consider multi-type TASEPs $\xi_t \in
\mathbb{Z}^{\mathbb{Z}}$ similar to Section 2.3. With the results from
Theorem \ref{hydro} we can calculate the distribution of the
asymptotic speed of a single second class particle in the TASEP with
initial configuration
$$\xi_0(x) =
\begin{cases}
	1 & x \leq -1 \\
	2 & x = 0 \\
	3 & x \geq 1
\end{cases}
$$ As the particles of class 3 are weaker than all other particles in
the model we can think of these particles as holes. So the second-class particle
sees only particles to its left and only holes to its right. 
The second-class particle can be seen as a discrepancy between two 
copies of the ``step'' initial condition considered in the last section,
one of which is shifted by one step to the right. Hence the path of 
the second-class particle corresponds to the propagation of the discrepancy
under the basic coupling.
The results for
the models in discrete time correspond to the result for the model in
continuous time first obtained in Ferrari and Kipnis
\cite{ferrarikipnis}. They prove convergence in distribution. In order
to prove a.s. convergence we can use the connection to last-passage
percolation and the growth model: As in Ferrari and Pimentel
\cite{ferraripimentel} the path of the second class particle
corresponds to a competition interface in the growth model which has
a.s. an asymptotic direction.

\begin{theorem} \label{2CP}
For $i=0,1,2,3$ let $X^{(i)}(t)$ denote the position of the second class particle at time $t$ in the corresponding model. Then we have
$$ \frac{X^{(i)}(t)}{t} \xrightarrow[t \rightarrow \infty]{a.s.} U^{(i)} $$
for random variables $U^{(i)}$ with distribution functions $1-f_i$ for $i=0,1,2$ and $a_3$ for $i=3$.
\end{theorem}

\begin{remark} \label{dirR1R3}
  The proofs for convergence in distribution are analogous to those
  for the model in continuous time (see for example
  \cite{ferrarikipnis}) apart from a small complication in the model
  R3 with the \textit{particle-particle coupling}. We will give an
  account of this proof in Section 4. For the almost sure
  convergence we will explain the construction of the competition
  interface for the models in discrete time and prove that the second
  class particle has almost sure an asymptotic speed by using results
  about semi-infinite geodesics in the percolation models similar to
  \cite{ferraripimentel}. An interesting observation will be that the
  distribution of the asymptotic direction of the competition interface
  corresponding to the path of the second class particle is the same
  in the models R1 and R3, see Remark \ref{proofdirR1R3}.  
However, this does not imply that
  the distribution of the speed of the second class particle is the same 
in the two models.
\end{remark}

\begin{remark}
In the continuous model the distribution of the asymptotic speed of
the second class particle turns out to be uniform on $[-1,1]$. The
distributions in the models in discrete time are more complicated.
\end{remark}

\bigskip

Using Theorem \ref{2CP} we can define the following so-called speed
process: Consider the multi-type TASEP with initial configuration $\xi_0(n) =
n$. By Theorem \ref{2CP} we know that each particle has a.s.\ an
asymptotic speed $U_n$: Particle $n$ has only stronger particles to
its left and only weaker particles (that can be seen as holes) to
its right just like the second class particle in the initial
configuration of Theorem \ref{2CP} and therefore we can apply Theorem
\ref{2CP} to the speed of every particle. We call the process $\left\{ U_n \right\}_{n
  \in \mathbb{Z}}$ the speed process and denote its distribution by
$\mu$. 
This process is stationary, and its marginals (for the various models) 
are given by the distributions in Theorem \ref{2CP}. 
Furthermore, we write $Y_n(m)$ instead of $\xi_m(n)$ and
denote the position of particle $n$ at time $m$ by $X_n(m)$ in order
to be consistent with the notation introduced by Amir, Angel and
Valk\'{o} \cite{amirangelvalko}. They have studied the speed process
for the model in continuous time. 
Note that both $Y_n$ and $X_n$ can be seen as permutations
of the set $\ZZ$, and are inverse to each other.
We define $g_i = 1 -
f_i$ for $i=0,1,2$ and $g_3 = 1 - a_3$, $g_4 = 1 - \frac{a_3 (1-\beta)}{1-a_3\beta}$
and our first result is the following theorem corresponding to Theorem
1.5 in \cite{amirangelvalko} (note that the labels of particles can
now be in $\mathbb{R}$ instead of just $\mathbb{Z}$):

\begin{theorem} \label{duality} For $i=0,1,2$, $\mu^{(i)}$, the
  distribution of the speed process in model R$i$, is the unique
  stationary ergodic measure for the TASEP R0 whose marginals
  have distribution function $g_i$. 
Correspondingly, for $i=3,4$, $\mu^{(i)}$ is the
  unique stationary measure for the TASEP R$j(i)$, which has marginals
  distributed according to $g_{i}$ on even sites and $g_{j(i)}$ on odd
  sites, where $j(3) = 4$ and $j(4) = 3$.
\end{theorem}

  For $i=0$ this gives the result from \cite{amirangelvalko} saying
  that the distribution of the speed process is itself a stationary
  ergodic measure for the TASEP in continuous time (the marginals are
  uniform on $\left[ -1,1 \right]$ in this case). The other parts of
  Theorem \ref{duality} follow from nice dualities between the models R1 and
  R2 and between the models R3 and R4, and the fact that R0, R1 and R2 
all have the same set of stationary distributions, whatever the value of $\beta$
(as given in Theorem \ref{invariantmulti}). 
The dualities are given by the following result: 

\begin{theorem} \label{duality2} Consider the starting configuration
  $Y_n(0) = n$. For $i=0,1,2,3,4$ and any fixed $m > 0$ the process
  $\{ X_{n}^{(i)}(m) \}_{n \in \mathbb{Z}}$ has the same distribution
  as the process $\{ Y_{n}^{(j(i))}(m) \}_{n \in \mathbb{Z}}$ where
$j(0)=0$, $j(1)=2$, $j(2)=1$, $j(3)=4$ and $j(4)=3$.
\end{theorem}

The following theorems provide some explicit results about the joint
distributions of the speeds of adjacent particles (and particles $0$
and $2$ in model R3). The first result is Theorem 1.7 in
\cite{amirangelvalko}. The remaining theorems and remarks give
analogous results for the TASEPs R1, R2 and R3.

\begin{theorem}[TASEP R0] \label{jointcont}
The joint distribution of $\left( U_0 , U_1 \right)$, supported on $\left[ -1 , 1 \right]^2$, is
$$ s(x,y) dx dy + r(x) \mathbbm{1}_{\left\{ x = y \right\}} dx $$
with
$$
s(x,y) =
\begin{cases}
	\frac{1}{4} & x > y \\
	\frac{y-x}{4} & x \leq y
\end{cases}
\qquad \text{and} \qquad r(x) = \frac{1-x^2}{8} $$
In particular, $\mathbb{P} \left[ U_0 > U_1 \right] = \frac{1}{2}$, $\mathbb{P} \left[ U_0 = U_1 \right] = \frac{1}{6}$ and $\mathbb{P} \left[ U_0 < U_1 \right] = \frac{1}{3}$.
\end{theorem}

\begin{theorem}[TASEP R1] \label{jointR1}
The joint distribution of $\left( U_0, U_1 \right)$ has support on $\left[ - \frac{\beta}{1-\beta} , \beta \right]^2$ and is given by
$$ s_1(x,y) dx dy + r_1(x) \mathbbm{1}_{\left\{ x = y \right\}} dx $$
with
$$ s_1(x,y) =
\begin{cases}
	\frac{1-\beta}{4\beta^2} \left( 1 - x \right)^{-\frac{3}{2}} \left( 1 - y \right)^{-\frac{3}{2}} = g_1^{'}(x)g_1^{'}(y) & x>y \\
	\frac{1-\beta}{2\beta^3} \left( 1 - x \right)^{-\frac{3}{2}} \left( 1 - y \right)^{-\frac{3}{2}} \left( \sqrt{\frac{1-\beta}{1-y}} - \sqrt{\frac{1-\beta}{1-x}} \right) & x \leq y
\end{cases}
$$
and
\begin{center}
$r_1(x) = \left( \frac{\sqrt{1-\beta}}{2 \beta^2 \left( 1 - u \right)^{\frac{3}{2}}} \left( 1 - \frac{1}{\beta} \right) + \frac{1-\beta}{2 \beta^2 \left( 1 - u \right)^{2}} \left( \frac{2}{\beta} - 1 \right) - \frac{\sqrt{1-\beta}\left( 1 - \beta \right)}{2 \beta^3 \left( 1 - u \right)^{\frac{5}{2}}} \right)$
\end{center}
In particular, $\mathbb{P} \left[ U_0 > U_1 \right] = \frac{1}{2}$, $\mathbb{P} \left[ U_0 = U_1 \right] = \frac{1}{6}$ and $\mathbb{P} \left[ U_0 < U_1 \right] = \frac{1}{3}$.
\end{theorem}


\begin{remark} By symmetry, the joint distribution of $\left(U_0,U_1\right)$ in the model R2
is the same as that of $\left(-U_1,-U_0\right)$ in the model R1.
\end{remark}

\begin{theorem}[TASEP R3] \label{jointR3a}
The joint distribution of $\left( U_0, U_1 \right)$ has support on $\left[ - \frac{2\beta}{2-\beta} , \frac{2\beta}{2-\beta} \right]^2$ and is given by
$$ s_2(x,y) dx dy + r_2(x) \mathbbm{1}_{\left\{ x = y \right\}} dx $$
with
$$ s_2(x,y) =
\begin{cases}
	g_3^{'}(x) g_{4}^{'}(y) & x>y \\
	g_3^{'}(x) g_{4}^{'}(y) \left( g_{4}(y) - g_{4}(x) \right) \\
	\qquad \qquad \cdot \left( 2 - g_{4}(x) \beta - g_{4}(y) \beta \right) \left( \frac{2 + y}{2 - y} \right) & x \leq y
\end{cases}
$$
and
\begin{center}
$r_2(x) = \frac{1-\beta}{\beta^3} \left( \frac{2(2-\beta)}{4 - x^2} - \frac{8}{4 - x^2} \sqrt{ \frac{1-\beta}{4-x^2}} \right)$
\end{center}
In particular,
\footnotesize
$$ \mathbb{P} \left[ U_0 > U_1 \right] = \frac{1}{\beta^2} \left( \beta - \left( 1 - \beta \right) \log \left( \frac{1}{1-\beta} \right) \right) $$
$$ \mathbb{P} \left[ U_0 = U_1 \right] = \frac{(1-\beta)(2-\beta)}{\beta^3} \log \left( \frac{1}{1-\beta} \right) - \frac{2(1-\beta)}{\beta^2} $$
\normalsize
and
\footnotesize
$$ \mathbb{P} \left[ U_0 < U_1 \right] = \frac{(1-\beta)(2-\beta)}{\beta^2} + \frac{2(1-\beta)^2}{\beta^3} \log(1-\beta) $$
\normalsize
\end{theorem}

\begin{remark}
Again by symmetry, we have that under rule R3, $(U_1,U_2)$ has the same distribution as $(-U_1,-U_0)$.
\end{remark}


\begin{theorem}[TASEP R3] \label{jointR3c}
The joint distribution of $\left( U_0, U_2 \right)$ has support on $\left[ - \frac{2\beta}{2-\beta} , \frac{2\beta}{2-\beta} \right]^2$ and is given by
$$ s_3(x,y) dx dy + r_3(x) \mathbbm{1}_{\left\{ x = y \right\}} dx $$
with
$$ s_3(x,y) =
\begin{cases}
	g_3^{'}(x) g_3^{'}(y) & x>y \\
	g_3^{'}(x) g_3^{'}(y) \left( g_{4}(x)^2 - g_{4}(y)^2 \vphantom{+ \frac{2(g_{4}(y) - 1)(g_{4}(x)^2 - 2g_{4}(x)g_{4}(y) + g_{4}(y) - 1}{1 - \beta g_{4}(y)} - 1} \right. \\
	\qquad - \frac{2(g_{4}(x) - 1)(g_{4}(y)^2 - 2g_{4}(x)g_{4}(y) + g_{4}(x) - 1}{1 - \beta g_{4}(x)} \\
	\qquad \qquad \left. + \frac{2(g_{4}(y) - 1)(g_{4}(x)^2 - 2g_{4}(x)g_{4}(y) + g_{4}(y) - 1}{1 - \beta g_{4}(y)} - 1 \right) & x \leq y
\end{cases}
$$
and
\begin{center}
$r_3(x) = \frac{g_3(u)(1-g_3(u))(1 - g_{4}(u)(1 - g_{4}(u)))}{\frac{(2-u)\beta}{2} \sqrt{\frac{4-u^2}{1-\beta}}}$
\end{center}
In particular,
\footnotesize
$$ \mathbb{P} \left[ U_0 > U_2 \right] = \frac{1}{2} $$
$$ \mathbb{P} \left[ U_0 = U_2 \right] = \frac{1}{6} + \frac{1}{3\beta} - \frac{13}{3\beta^2} + \frac{8}{\beta^3} - \frac{4}{\beta^4} - \frac{\left( 1 - \beta \right)^2}{\beta^3} \left( \log \left( \frac{1}{1-\beta} \right) \right) \left( \frac{2}{\beta} - \frac{4}{\beta^2} \right) $$
\normalsize
and
\footnotesize
$$ \mathbb{P} \left[ U_0 < U_2 \right] = \frac{1}{3} - \frac{1}{3\beta} + \frac{13}{3\beta^2} - \frac{8}{\beta^3} + \frac{4}{\beta^4} + \frac{\left( 1 - \beta \right)^2}{\beta^3} \left( \log \left( \frac{1}{1-\beta} \right) \right) \left( \frac{2}{\beta} - \frac{4}{\beta^2} \right) $$
\normalsize
\end{theorem}

We see that in every model we have
  that the speeds are independent on the set where $U_0 > U_1$ ($U_0 >
  U_2$ respectively). This agrees with the result in continuous
  time. The striking result, shown in \cite{amirangelvalko} for the continuous model, that with
  positive probability the two continuous random variables $U_0$ and
  $U_1$ are equal, holds also in the discrete models.

Interestingly, the probabilities $\mathbb{P} \left[ U_0 > U_1
  \right]$, $\mathbb{P} \left[ U_0 = U_1 \right]$ and $\mathbb{P}
\left[ U_0 < U_1 \right]$ are the same for models R0, R1 and R2, and
do not depend on the
parameter $\beta$. This is rather surprising since $\beta$ is not
just a scaling parameter (i.e. we cannot produce models with different
values of $\beta$ by just applying a time change). In fact, much more 
is true. From the first part of Theorem \ref{invariantmulti},
we see that, although the marginal distribution of each $U_i$ 
depends on the model and the value of $\beta$, we can obtain the 
distribution for either of R1 and R2 and any value of $\beta$
by applying an appropriate monotone function to each entry $U_i$
(see the proof of Theorem \ref{jointR1} for further details).
Hence the relative ordering of the variables $U_i$ is not affected
by the model or the value of $\beta$.

To go further, consider particles $i$ and $j$ with $i<j$.
It's clear that if $U_i<U_j$ then 
particle $i$ can never overtake particle $j$, while if $U_i>U_j$
then particle $i$ must overtake particle $j$. 
In \cite{amirangelvalko}, it's shown that for the continuous-time model, 
with probability 1, if $U_i=U_j$ then particle $i$ overtakes particle $j$. 
The same result can be shown for the discrete-time models, although 
the calculations involved in the argument are rather more complicated than those 
used to prove Theorem 1.14 of \cite{amirangelvalko}, and we omit them here.
So, for example, the probability that particle $i$ overtakes particle
$j$ is the same for models R0, R1 and R2. Indeed, more completely
one can define an ordering $\prec$ on $\mathbb{Z}$ by $i\prec j$ iff 
particle $j$ is eventually to the right of particle $i$. Then we have the following result:
\begin{cor}\label{orderingcorollary}
The ordering $\prec$ has the same distribution for R0, R1 and R2 and for any
value of $\beta$. 
\end{cor}
It would certainly be interesting to have a more direct understanding
of this property, based for example on couplings or local dynamics, as
well as the indirect argument based on the equivalence of multi-type
equilibrium distributions. 

Overtaking probabilities in the multi-type TASEP can also be 
interpreted in terms of questions of survival or extinction
in multi-type growth models. In \cite{ferrarigoncalvesmartin}
a coupling is given between the multi-type TASEP and 
a three-type version of the corner growth model, under
which a given cluster survives for ever if and only if
particle 0 never overtakes particle 1. (The extinction
of the cluster occurs if two interfaces in the growth model 
meet -- the paths of these interfaces are related to the paths
of the two particles). Different overtaking events 
in the TASEP can be represented by varying the initial condition
in the competition growth model. Using the results above, 
we find that the survival probabilities in the growth model
will remain unchanged if we move from the continuous-time model
to natural discrete-time models which correspond to models R1 or R2 in the 
TASEP. Again, this is certainly 
not obvious from the local dynamics of the processes. 

Unlike in models R1 and R2, 
in the model R3 the probabilities $\mathbb{P} \left[
  U_0 > U_1 \right]$, $\mathbb{P} \left[ U_0 = U_1 \right]$ and
$\mathbb{P} \left[ U_0 < U_1 \right]$ do depend on $\beta$ and the
behaviour of the model is qualitatively different for
different values of $\beta$ (Theorem \ref{jointR3a}): For small
$\beta$ we have $\mathbb{P} \left[ U_1 < U_0 \right] > \mathbb{P}
\left[ U_1 > U_0 \right]$, but $\mathbb{P} \left[ U_1 < U_0 \right] <
\mathbb{P} \left[ U_1 > U_0 \right]$ for large $\beta$ (the transition
occurs at $\beta = 0.38860064568\ldots$).

Note however that for $\beta \rightarrow 1$ the probabilities 
relating $U_0$ and $U_2$
in Theorem
\ref{jointR3c} converge to $\frac{1}{2}$, $\frac{1}{6}$ and
$\frac{1}{3}$, i.e.\ to the probabilities we get in the continuous model
and R1 and R2 for the speeds of particle $0$ and $1$. In a sense, for
large $\beta$ the particles $0$ and $2$ in the model R3 behave like
adjacent particles in the models R1 and R2. This can heuristically be
seen in the following way: We consider the particles in the model R3
(with large $\beta$ close to $1$) starting on even sites. In general,
particles starting on an even site will move two steps to the right in
each time-step since $\beta$ is large and we update even sites
first. If a particle does not jump either during the even or the odd
update (which happens with probability $2 \beta (1 - \beta)$) it ends
up on an odd site and starts moving left until either
\begin{itemize}
 \item (A) it hits a weaker particle to the left by which it cannot be overtaken
 \item (B) it does not get jumped over either during an even or odd
   update because the adjacent particle to the left did not try to
   jump
\end{itemize}
In both cases the particle itself will return to an even site (with
high probability) and resume moving to the right. The particle that
caused the stop (either because it was weaker or because it did not
try to jump) will itself start moving to the left until (A) or (B)
happens. Now consider the model R1 with large $\beta$. Most particles
will move one step to the right in each time-step, but some particles
do not jump and therefore get overtaken until again either (A) or (B)
happens (where we remove the part ``either during and even or odd
update''). Particles in these two models have different speeds, but
the probabilities $\mathbb{P} \left[ U_0 > U_1 \right]$, $\mathbb{P}
\left[ U_0 = U_1 \right]$, $\mathbb{P} \left[ U_0 < U_1 \right]$ in R1
and $\mathbb{P} \left[ U_0 > U_2 \right]$, $\mathbb{P} \left[ U_0 =
  U_2 \right]$ and $\mathbb{P} \left[ U_0 < U_2 \right]$ in R3 are
(almost) the same.


\section{Proofs}

\subsection{Invariant Measures}

As the idea of the proof for Theorem \ref{invariantmulti} is the same for the discrete-time models as for the continuous-time model R0 we will only sketch the proof. When thinking about the model R3 bear in mind that we have different densities on even and odd sites.


\begin{proof}[Proof sketch for Theorem \ref{invariantmulti}:]
We can proceed in the same way as in \cite{ferrarimartin}: Using arguments as in \cite{mountfordprabhakar} we can see that for every parameter $\rho \in (0,1)$ there exists an essentially unique function $H_{\rho}$ which maps Bernoulli processes $\omega$ on $\mathbb{Z} \times \mathbb{Z}$ onto stationary and space-ergodic doubly infinite trajectories $(\eta_n)_{n \in \mathbb{Z}}$ of the TASEP governed by $\omega$ with time-marginals $\mu_{\rho}$ (Proposition 8 in \cite{ferrarimartin}). For each $\rho$ and $\omega$ we can construct a set of dual points $\bigtriangleup_{\rho}(\omega)$ which govern the time-reversal of $(\eta_n)_{n \in \mathbb{Z}}$ and again form a Bernoulli process. Also the set of dual points before time $m$ is independent of the configuration $\eta_m$ (Proposition 10 in \cite{ferrarimartin}). We now take $\rho_1 < \ldots < \rho_n$ and let $\alpha_m = (\alpha^1_m, \ldots, \alpha^n_m)$ be the multiline TASEP trajectory governed by $\omega$. This means that $\omega^n = \omega$, $\omega^k = \bigtriangleup_{\rho_{k+1}}(\omega^{k+1})$ and $(\alpha^k_m)_{m \in \mathbb{Z}}$ is the TASEP trajectory governed by $\omega^k$ with density $\rho_k$. Then by the independence of the dual points before time $m$ from the configuration $\eta_m$ we get that the multiline process is stationary with product measure $\nu = \nu_{\rho_1} \times \ldots \times \nu_{\rho_n}$ (Proposition 11 in \cite{ferrarimartin}). As in the paragraph preceding Theorem \ref{invariantmulti} we define $\eta = \left( \eta^1 , \ldots , \eta^n \right)$ by
$\eta^k = D^{(n-k+1)} \left( \alpha_k , \ldots , \alpha_n \right)$. Then induction arguments and some case-by-case checking for $n=2$ show that $(\eta^k_n)_{n \in \mathbb{Z}}$ is the TASEP trajectory governed by $\omega$ with particle density $\rho_k$ (Proposition 12 in \cite{ferrarimartin}) and this implies Theorem \ref{invariantmulti}.
\end{proof}

\begin{remark}
As mentioned in the beginning of this section, for the model R3 we have to think of the $\rho_k$ as densities on even sites and we have to replace the $\nu_{\rho_k}$ by $\mu_{\rho_k}$.
\end{remark}

\begin{remark}
Inherent in the tandem queue construction for the multi-type
stationary distribution in model R3
is a version of Burke's theorem for 
the queues with different arrival
  and service rates on even and odd sites.
  Consider a queue with arrival process $A_n$, service process
  $S_n$ and departure process $D_n$. Let $A_n$ be a Bernoulli process
  with rate $\rho_1 = (\gamma_1, \gamma_2) \in (0,1)^2$ which means
  that on even sites arrivals happen with probability $\gamma_1$ and
  on odd sites they happen with probability $\gamma_2$. Motivated by
  the invariant distributions for the TASEP R3 with just one type of
  particles we want $\gamma_1$ and $\gamma_2$ to satisfy
\begin{align}
\gamma_2 = \frac{\gamma_1 \left( 1 - \beta \right)}{1 - \gamma_1 \beta} \label{rho1}
\end{align}
where $\beta \in \left( 0,1 \right)$ is the rate at which jumps in the
TASEP happen. Analogously, we let $S_n$ be a Bernoulli process with
rate $\rho_2 = ( \delta_1, \delta_2 ) \in (0,1)^2$ where
\begin{align}
\delta_2 = \frac{\delta_1 \left( 1 - \beta \right)}{1 - \delta_1 \beta} \label{rho2}
\end{align}
(and $\gamma_1 < \delta_1$, $\gamma_2 < \delta_2$). The main
observation in Burke's Theorem (see for example \cite{burke}) that
shows that arrival and departure process have the same distribution is
that the queue length process is reversible and that departures look
in the reversed process like arrivals in the original
process. Interestingly, it turns out that in the queueing model
described above there exists a stationary reversible distribution
$\pi$ for the queue length process which is independent of whether we
just observed arrivals and services at even sites or at odd
sites. $\pi$ is given by
$$ \pi(j) = \left( 1 - \frac{\gamma_1 \left( 1 - \delta_1 \right)}{\left( 1 - \gamma_1 \right) \delta_1} \right) \left( \frac{\gamma_1 \left( 1 - \delta_1 \right)}{\left( 1 - \gamma_1 \right) \delta_1} \right)^j \qquad j = 0,1,\ldots $$
and it is reversible because it satisfies the two systems of detailed
balance equations
$$ \pi(j) \gamma_i \left( 1 - \delta_i \right) = \pi(j+1) \left( 1 - \gamma_i \right) \delta_i \qquad j=0,1,\ldots $$
for $i=1,2$ ($i=1$ corresponds to even sites, $i=2$ corresponds to odd
sites). This follows from the relations (\ref{rho1}) and
(\ref{rho2}). As in Burke's Theorem it follows from the reversibility
of the queue length process that the departure process has the same
distribution as the arrival process, i.e. $D_n$ is a Bernoulli process
with rate $\rho_1 = (\gamma_1, \gamma_2)$.

Indeed, the multi-type construction yields extensions
of this result which give input-output theorems for 
the priority queues with more than one type of customer. 
For a discussion of the analogous
result in the context of constant arrival and service rates,
see for example Section 6 of \cite{ferrarimartin}.
\end{remark}

\subsection{Hydrodynamic limits}

\begin{proof}[Proof outline for Theorem \ref{hydro} for the TASEP R3:]
The following Propositions correspond to Propositions 2,3 and 5 in \cite{rost} and the proofs are essentially the same as in \cite{rost}.
\begin{proposition} \label{convS3}
For all $u \in \mathbb{R}$ the random variables $\frac{1}{n}S(\left[un\right],n) = \frac{1}{n} \sum_{k > [un]} \eta_n(k)$ converge a.s. and in $L^1$ to a constant $h_3(u)$, as $n$ goes to infinity. The function $h_3$ is decreasing, convex; one has $h_3(u) = 0$ for $u > \frac{2\beta}{2-\beta}$ and $h_3(u) = -u$ for $u < - \frac{2\beta}{2-\beta}$.
\end{proposition}
\begin{proposition} \label{densityprofile3}
If $h_3$ is differentiable at $u$, one has
$$ \lim_{n \rightarrow \infty} \mathbb{E} \left[ \eta_n \left( 2 \left[ \frac{k}{2} \right] \right) + \eta_n \left( 2 \left[ \frac{k}{2} \right] + 1 \right) \right] = - 2 h_3^{'}(u)  $$
whenever $\frac{k}{n}$ tends to $u$.
\end{proposition}
\begin{proposition} \label{weakconv3} Let $\mu_3(k,n)$ be the distribution of
  $(\eta_{n}(k+l), l \in \mathbb{Z})$. If $h_3^{'}(u)$ exists, any
  weak limit $\mu_3^*$ of the measures 
$\mu_3 \left( 2 \left[\frac{un}{2} \right] , n \right)$ for $n \rightarrow \infty$ is
  of the form
$$ \mu_3^* = \int_0^1 \tau_x \sigma(dx) $$
with some probability $\sigma$ on $\left[ 0, 1 \right]$. $\tau_x$ is the Bernoulli product measure with density $b(x)$ on even sites and density $\frac{b(x) \left( 1 - \beta \right)}{1 - b(x) \beta}$ on odd sites where $b(x)$ is such that the average density is given by $x = \frac{1}{2} \left( b(x) + \frac{b(x) \left( 1 - \beta \right)}{1 - b(x) \beta} \right)$. That means that from Proposition \ref{densityprofile3} it follows that the measure $\sigma$ satisfies
$$ \int_0^1 x \sigma(dx) = f_3(u) = - h_3^{'}(u) $$
\end{proposition}
We can use the results from O'Connell \cite{oconnell} about last-passage percolation (see (\ref{LPPgeom})) to calculate the function $h_3$:
$$ h_3(u) =
\begin{cases}
	-u & u \leq - \frac{2 \beta}{2 - \beta} \\
	\frac{1}{\beta} \left( 2 - \frac{u \beta}{2} - \sqrt{ \left( 4 - u^2 \right) \left( 1 - \beta \right) } \right) - 1 & - \frac{2 \beta}{2 - \beta} \leq u \leq \frac{2 \beta}{2 - \beta} \\
	0 & u \geq \frac{2 \beta}{2 - \beta}
\end{cases}
$$
Since $h_3$ is differentiable we can identify $f_3 = - h_3^{'}$ from Proposition \ref{weakconv3} as
$$ f_3(u) = -h_3^{'}(u) =
\begin{cases}
	1 & u \leq - \frac{2 \beta}{2 - \beta} \\
	\frac{1}{2} - \frac{u}{\beta} \sqrt{ \frac{1 - \beta}{4 - u^2}} & - \frac{2 \beta}{2 - \beta} \leq u \leq \frac{2 \beta}{2 - \beta} \\
	0 & u \geq \frac{2 \beta}{2 - \beta}
\end{cases}
$$
As mentioned after Theorem \ref{hydro} in the models R0, R1 and R2
this is enough to prove the convergence statements at the end of
Theorem \ref{hydro}; this is done using the monotonicity of 
the distribution of $\eta_n(k)$ in $k$. 
However, due to the different behaviour at odd and even sites,
this monotonicity does not hold for the model R3, 
and the average density
$f_3$ does not pick up the density fluctuations between even and odd
sites. Hence without knowing that the model converges to local
equilibrium (which would allow us to calculate the function $a_3$ from
$f_3$) we cannot prove the last part of Theorem \ref{hydro}. The
essential step for proving convergence to local equilibrium is the
following proposition, the proof of which is again the same as in
\cite{rost} (Proposition 6). 
Let $\rho(k;F;n) = \mathbb{P} \left[ \eta_n(k + i) =
1 , i \in F \right] $ for a set $F \subset \mathbb{Z}$.
\begin{proposition} \label{mu3ineq2}
For any finite set $F$ and any $\epsilon > 0$ there exists a $\delta > 0$, $n_0$ such that
$$ \left| \rho\left(2\left[\frac{un}{2} \right];F;n\right) 
- \rho\left(2\left[\frac{\overline{u}n}{2} \right];F;n\right) \right| 
\leq \epsilon $$
for $\left| u - \overline{u} \right| \leq \delta$ and $n \geq
n_0$. Also
$$ \left| \rho\left(2 \left[\frac{\overline{u}n}{2} \right];F;n+l\right) 
- \rho\left(2 \left[\frac{\overline{u}n}{2} \right];F;n\right) \right| \leq \epsilon $$
for $0 \leq l \leq [\delta n]$, $n \geq n_0$.
\end{proposition}

Using this proposition and Jensen's inequality we can prove that the
measure $\sigma$ from Proposition \ref{weakconv3} is the unit mass on
$f_3(u)$ and since $b(f_3(u)) = a_3(u)$ this implies convergence to
local equilibrium (see \cite{rost}, Section 4 for the details).
\end{proof}

\subsection{Multi-type models out of equilibrium}

\begin{proof}[Proof of Theorem \ref{2CP}:]
First we want to outline the proof for convergence of $\frac{X^{(i)}(t)}{t}$ in distribution (as before we have $t \in \mathbb{R}_+$ or $t \in \mathbb{N}$ depending on the model). This follows the ideas in \cite{ferrarikipnis}. We want to couple two TASEPs with initial configurations
$$
\eta_0^1(x) =
\begin{cases}
	1 & x \leq 0 \\
	0 & x \geq 1
\end{cases}
\text{ and }
\eta_0^2(x) =
\begin{cases}
	1 & x \leq -1 \\
	0 & x \geq 0
\end{cases}
$$
in two different ways and calculate the difference $\mathbb{E} \left[ S^1([rt],t) \right] - \mathbb{E} \left[ S^2([rt],t) \right]$ in both couplings. $S^1([rt],t)$ and $S^2([rt,t])$ are the number of particles to the right of $[rt]$ at time $t$ in $\eta^1$ and $\eta^2$. Using basic coupling (i.e. using the same Poisson processes $\{ \left( P^x_t
\right)_{t \geq 0} : x \in \mathbb{Z} \}$ or Bernoulli processes $\{ \left( B_n^x \right)_{n \in \mathbb{N}} : x \in \mathbb{Z} \}$ for $\eta^1$ and $\eta^2$) gives
\begin{equation}
\mathbb{E} \left[ S^1([rt],t) \right] - \mathbb{E} \left[ S^2([rt],t) \right] = \mathbb{P} \left[ X^{(i)}(t) > [rt] \right] \label{BC}
\end{equation}
since we can interpret the discrepancy between $\eta^1$ and $\eta^2$ as second class particle (see Section 2.3). This works for all models R0, R1, R2 and R3. The second coupling we want to use is called particle-particle coupling. We label the particles in $\eta^1_0$ and $\eta^2_0$ from right to left and let particles with the same label jump at the same time. Then under this coupling
\begin{equation}
\mathbb{E} \left[ S^1([rt],t) \right] - \mathbb{E} \left[ S^2([rt],t) \right] = \mathbb{P} \left[ \eta_t^1([rt]+1) = 1 \right] \label{PPC}
\end{equation}
in the models R0, R1 and R2. By Theorem \ref{hydro} the right hand side of (\ref{PPC}) converges to $f_i(r)$, so together with (\ref{BC}) we have
$$ \mathbb{P} \left[ X^{(i)}(t) > [rt] \right] \xrightarrow[t \rightarrow \infty]{} f_i(r) $$
which proves convergence in distribution for $i=0,1,2$. However, in the model R3 we cannot use the particle-particle coupling as before because this is no longer a real coupling. If we let particles with the same label jump at the same time in $\eta^1$ and $\eta^2$ then the dynamics of the $\eta^2$ process are different from the $\eta^1$ process: In $\eta^2$ we update odd sites first and then even sites. We denote by $\mathbb{E}_{PP}$ the expectation in $\eta^1$ and $\eta^2$ if particles with the same label jump at the same time in $\eta^1$ and $\eta^2$ where we update in such a way that $\eta^1$ is still a TASEP with update rule R3. Then we have
$$ \mathbb{E}_{PP} \left[ S^1([rt],t) \right] = \mathbb{E} \left[ S^1([rt],t) \right] $$
Notice that starting the second process with updating even sites does not change anything as there is no particle on an even site with an adjacent empty site in the initial configuration. If we remove the last update of even sites at time $t$ this changes the value of $S^2([rt],t)$ if there is a jump from site $[rt]$ to site $[rt] + 1$ during this update. But there can only be a jump from $[rt]$ to $[rt] + 1$ while updating the even sites if $[rt]$ is even. Let us therefore consider odd sites $2 [ \frac{rt}{2} ] + 1$ first.
We get
\begin{equation} \label{EPP}
\mathbb{E}_{PP} \left[ S^2 \left( 2 \left[ \frac{rt}{2} \right] + 1,t \right) \right] = \mathbb{E} \left[ S^2\left( 2 \left[ \frac{rt}{2} \right] + 1,t \right) \right]
\end{equation}
and we still have
$$ \mathbb{E}_{PP} \left[ S^1\left( 2 \left[ \frac{rt}{2} \right] + 1,t \right) \right] - \mathbb{E}_{PP} \left[ S^2\left( 2 \left[ \frac{rt}{2} \right] + 1,t \right) \right] = \mathbb{P} \left[ \eta^1_t \left( 2 \left[ \frac{rt}{2} \right] + 2 \right) = 1 \right] $$
as before. Hence we get
\begin{align*}
\mathbb{P} \left[ X^{(3)}(t) > 2 \left[ \frac{rt}{2} \right] + 1 \right]
&= \mathbb{E} \left[ S^1\left( 2 \left[ \frac{rt}{2} \right] + 1,t \right) \right] - \mathbb{E} \left[ S^2\left( 2 \left[ \frac{rt}{2} \right] + 1,t \right) \right] \\
&= \mathbb{P} \left[ \eta^1_t \left( 2 \left[ \frac{rt}{2} \right] + 2 \right) = 1 \right] \\
&\xrightarrow[t \rightarrow \infty]{} a_3(r)
\end{align*}
by the convergence to local equilibrium (Theorem \ref{hydro}).
But by monotonicity we have
$$ \mathbb{P} \left[ X^{(3)}(t) > 2 \left[ \frac{rt}{2} \right] - 1 \right] \geq \mathbb{P} \left[ X^{(3)}(t) > 2 \left[ \frac{rt}{2} \right] \right] \geq \mathbb{P} \left[ X^{(3)}(t) > 2 \left[ \frac{rt}{2} \right] + 1 \right] $$
Hence
$$ \lim_{t \rightarrow \infty} \mathbb{P} \left[ \frac{X^{(3)}(t)}{t} > r \right] = a_3(r) $$

\begin{remark}
If the second class particle starts on an odd site (with first class particles to the left and holes/third class particles to the right) then
$$ \frac{\widetilde{X}^{(3)}(t)}{t} \xrightarrow[t \rightarrow \infty]{a.s.} \widetilde{U}^{(3)} $$
and $\widetilde{U}^{(3)}$ has distribution function $\frac{a_3(1-\beta)}{1-a_3\beta}$ accordingly.
\end{remark}

Now we want to prove almost sure convergence of the speed of a second
class particle. Our methods follow the approach in
\cite{ferraripimentel}. The idea is to establish a connection between
the path of the second class particle and a competition interface in
the corresponding growth model. The cluster in the growth model can be
divided into two clusters corresponding to events happening to the
right and to the left of the second class particle. The interface
between these two clusters is called the competition interface. Using
results about semi-infinite geodesics it can be shown that this
competition interface has almost surely an asymptotic direction. This
can be used to deduce that the second class particle has almost surely
an asymptotic speed and since we know the distribution of this speed
we can also calculate the distribution of the random angle of the
competition interface. In the following we will describe this method
first for the TASEP in continuous time, as given in \cite{ferraripimentel}, 
and then explain the
adjustments that have to be made in the TASEPs in discrete time.
\begin{figure}[t]
\begin{center}
\begin{pspicture}(-3.5,-4.5)(9,1)
	\psline{->}(-3.5,0)(2.5,0)
	\psdots(-3,0)(-2,0)(-1,0)(0,0)(1,0)(2,0)
	\cnode(-3,0.5){0.3}{A}
	\cnode(-2,0.5){0.3}{A}
	\cnode(0,0.5){0.3}{A}
	\uput{0.2}[270](-3,0){-2}
	\uput{0.2}[270](-2,0){-1}
	\uput{0.2}[270](-1,0){0}
	\uput{0.2}[270](0,0){1}
	\uput{0.2}[270](1,0){2}
	\uput{0.2}[270](2,0){3}
	\psarc{-}(-0.5,0.5){0.9}{155}{205}
	\psarc{-}(-0.5,0.5){0.9}{-25}{25}
	\uput[0](0.3,0.7){*}
	\psline{->}(3,0)(9,0)
	\psdots(3.5,0)(4.5,0)(5.5,0)(6.5,0)(7.5,0)(8.5,0)
	\cnode(3.5,0.5){0.3}{A}
	\cnode(4.5,0.5){0.3}{A}
	\cput(5.5,0.5){2}
	\uput{0.2}[270](3.5,0){-2}
	\uput{0.2}[270](4.5,0){-1}
	\uput{0.2}[270](5.5,0){0}
	\uput{0.2}[270](6.5,0){1}
	\uput{0.2}[270](7.5,0){2}
	\uput{0.2}[270](8.5,0){3}	
	\psline{->}(-3.5,-2)(2.5,-2)
	\psdots(-3,-2)(-2,-2)(-1,-2)(0,-2)(1,-2)(2,-2)
	\cnode(-3,-1.5){0.3}{A}
	\cnode(-2,-1.5){0.3}{A}
	\cnode(1,-1.5){0.3}{A}
	\uput{0.2}[270](-3,-2){-2}
	\uput{0.2}[270](-2,-2){-1}
	\uput{0.2}[270](-1,-2){0}
	\uput{0.2}[270](0,-2){1}
	\uput{0.2}[270](1,-2){2}
	\uput{0.2}[270](2,-2){3}
	\psarc{-}(0.5,-1.5){0.9}{155}{205}
	\psarc{-}(0.5,-1.5){0.9}{-25}{25}
	\uput[0](1.3,-1.3){*}
	\psarc{<-}(0.5,-1.5){0.5}{40}{140}
	\uput{0.2}[180](-3.5,-1.9){A}
	\psline{->}(3,-2)(9,-2)
	\psdots(3.5,-2)(4.5,-2)(5.5,-2)(6.5,-2)(7.5,-2)(8.5,-2)
	\cnode(3.5,-1.5){0.3}{A}
	\cnode(4.5,-1.5){0.3}{A}
	\cput(6.5,-1.5){2}
	\uput{0.2}[270](3.5,-2){-2}
	\uput{0.2}[270](4.5,-2){-1}
	\uput{0.2}[270](5.5,-2){0}
	\uput{0.2}[270](6.5,-2){1}
	\uput{0.2}[270](7.5,-2){2}
	\uput{0.2}[270](8.5,-2){3}
	\psarc{<-}(6,-1.5){0.5}{40}{140}
	\psline{->}(-3.5,-4)(2.5,-4)
	\psdots(-3,-4)(-2,-4)(-1,-4)(0,-4)(1,-4)(2,-4)
	\cnode(-3,-3.5){0.3}{A}
	\cnode(-1,-3.5){0.3}{A}
	\cnode(0,-3.5){0.3}{A}
	\uput{0.2}[270](-3,-4){-2}
	\uput{0.2}[270](-2,-4){-1}
	\uput{0.2}[270](-1,-4){0}
	\uput{0.2}[270](0,-4){1}
	\uput{0.2}[270](1,-4){2}
	\uput{0.2}[270](2,-4){3}
	\psarc{-}(-1.5,-3.5){0.9}{155}{205}
	\psarc{-}(-1.5,-3.5){0.9}{-25}{25}
	\uput[0](-0.7,-3.3){*}
	\psarc{<-}(-1.5,-3.5){0.5}{40}{140}
	\uput{0.2}[180](-3.5,-3.9){B}
	\psline{->}(3,-4)(9,-4)
	\psdots(3.5,-4)(4.5,-4)(5.5,-4)(6.5,-4)(7.5,-4)(8.5,-4)
	\cnode(3.5,-3.5){0.3}{A}
	\cnode(5.5,-3.5){0.3}{A}
	\cput(4.5,-3.5){2}
	\uput{0.2}[270](3.5,-4){-2}
	\uput{0.2}[270](4.5,-4){-1}
	\uput{0.2}[270](5.5,-4){0}
	\uput{0.2}[270](6.5,-4){1}
	\uput{0.2}[270](7.5,-4){2}
	\uput{0.2}[270](8.5,-4){3}
	\psarc{<->}(5,-3.5){0.5}{40}{140}
\end{pspicture}
\end{center}
\setcaptionmargin{1cm}
\caption{Pair representation of the second class particle: The figures on the left show the system with the pair; the figures on the right show the corresponding multi-type system} \label{HP2CP}
\end{figure}
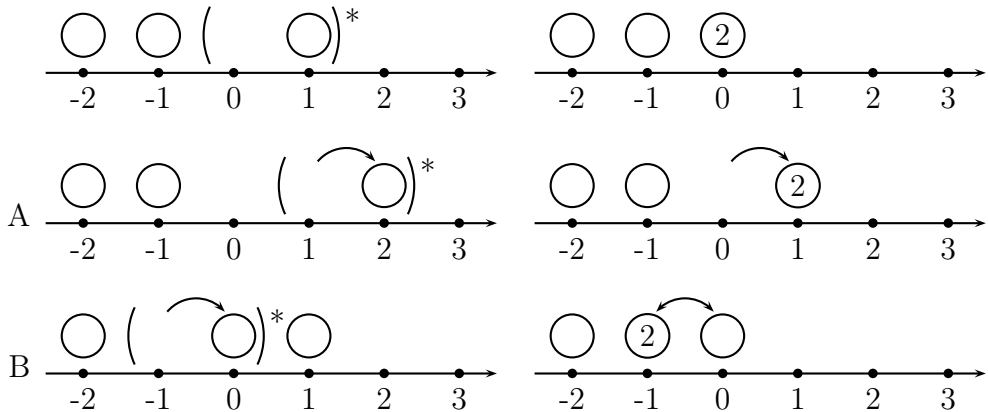
\\
In order to establish the connection between the second class particle and the competition interface we represent the second class particle as a pair consisting of a hole and a particle. This reduces our multi-type model to a model consisting only of particles and holes and allows us to use the connection to last-passage percolation developed in Section 2.2. We let the pair move as follows: If the particle of the pair jumps to the right the pair moves to the right (A) and if a particle jumps from the left into the hole of the pair the pair moves to the left (B), see Figure \ref{HP2CP}. Then the pair behaves indeed like a second class particle.
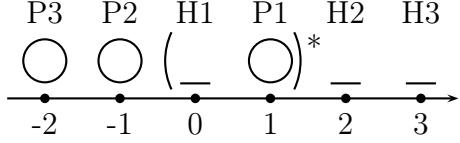
\begin{figure}[t]
\begin{center}
\begin{pspicture}(-4,-0.5)(3,1.5)
	\psline{->}(-3.5,0)(2.5,0)
	\psdots(-3,0)(-2,0)(-1,0)(0,0)(1,0)(2,0)
	\cnode(-3,0.5){0.3}{A}
	\cnode(-2,0.5){0.3}{A}
	\cnode(0,0.5){0.3}{A}
	\psline(-1.2,0.2)(-0.8,0.2)
	\psline(0.8,0.2)(1.2,0.2)
	\psline(1.8,0.2)(2.2,0.2)
	\uput{0}[90](-1,1){H1}
	\uput{0}[90](1,1){H2}
	\uput{0}[90](2,1){H3}
	\uput{0}[90](0,1){P1}
	\uput{0}[90](-2,1){P2}
	\uput{0}[90](-3,1){P3}
	\uput{0.2}[270](-3,0){-2}
	\uput{0.2}[270](-2,0){-1}
	\uput{0.2}[270](-1,0){0}
	\uput{0.2}[270](0,0){1}
	\uput{0.2}[270](1,0){2}
	\uput{0.2}[270](2,0){3}
	\psarc{-}(-0.5,0.5){0.9}{155}{205}
	\psarc{-}(-0.5,0.5){0.9}{-25}{25}
	\uput[0](0.3,0.7){*}
\end{pspicture}
\end{center}
\setcaptionmargin{1cm}
\caption{Pair representation of the second class particle with particles labelled from right to left and holes labelled from left to right} \label{HP2CP2}
\end{figure}
If we label the particles from right to left and the holes from left to right as in Figure \ref{HP2CP2} then we can consider the process
  $(\varphi_n)_{n \geq 0}$ giving the labels of the pair after the
  $n$th jump involving the pair. We have $\varphi_0 = (1,1)$, as initially the pair consists of hole 1 and particle 1, and
  $\varphi_{n+1} - \varphi_n \in \{ (0,1) , (1,0) \}$. $\varphi_n$ satisfies the recursion formula
\begin{equation}
\varphi_{n+1}=
\begin{cases}
 \varphi_n + (1,0) & T(\varphi_n + (1,0)) < T(\varphi_n + (0,1)) \\
 \varphi_n + (0,1) & T(\varphi_n + (1,0)) > T(\varphi_n + (0,1))
\end{cases} \label{phi}
\end{equation}
If the first
  label increases the second class particle moves one step to the right
  and if the second label increases the second class particle moves
  one step to the left.
\begin{remark}
Note that by inserting an extra site we changed the parity of the sites to the right of and including the particle of the pair: The first hole to the right of the second class particle is at an odd site while the first hole to the right of the pair is at an even site. The second class particle itself is at an even site while the particle in the pair is at an odd site. This will be important for model R3.
\end{remark}

Now we want to define geodesics in the last-passage percolation model
to define a competition interface in the growth model. For $z, z' \in
\ZZ_+^2$ the heaviest increasing path from $z$ to $z'$ (i.e. the path
that achieves the maximum in $R(z,z')$) is called the geodesic from
$z$ to $z'$. Note that in the model in continuous time geodesics are
unique. (In the models in discrete time we will need a rule to break
ties to achieve uniqueness of the geodesics). A semi-infinite
geodesic starting at $z$ is a path $\pi = (z,z_1,z_2,\ldots)$ in
$\ZZ_+^2$ such that for every $z'=z_k, z''=z_l \in \pi$ the geodesic
from $z'$ to $z''$ is $(z_k,z_{k+1},z_{k+2},\ldots,z_l) \subset
\pi$. For $\alpha \in \left[0,90\text{\textdegree}\right]$ a
$\alpha$-geodesic is a semi-infinte geodesic with direction
$\alpha$. Now we colour every block $Q(i,j) = \left( \left. i-1, i
\right] \right. \times \left( \left. j-1, j \right] \right.$ in
    $(\mathbb{R}_+)^2 \backslash [0,1]^2$ either red if the geodesic
    from $(1,1)$ to $(i,j)$ passes through $(1,2)$ or blue if it
    passes through $(2,1)$. The interface between these clusters is
    called the competition interface and an induction argument
    together with the recursion (\ref{phi}) shows that it is given by
    the process $\varphi_n$, see Proposition 3 in
    \cite{ferraripimentel}. Using results about the existence and
    uniqueness of $\alpha$-geodesics (Propositions 7,8 and 9 in
    \cite{ferraripimentel}) it can be shown that $\varphi_n$ has
    almost surely an asymptotic direction and we can conclude that the second
    class particle has almost surely an asymptotic speed (Propositions
    4 and 5 in \cite{ferraripimentel}). Now we want to apply these
    methods to the discrete time TASEPs R1,R2 and R3.

\subsubsection*{R1}
The last-passage percolation model for rule R1 was described
in Section \ref{percrep}, and in particular just after 
(\ref{T1}).
Now to adapt to
the initial configuration
$$\eta_0(x) =
\begin{cases}
	1 & x \leq -1 \\
	0 & x = 0 \\
	1 & x = 1 \\
	0 & x \geq 2
\end{cases}
$$ we have to remove the `$+1$' weight from the edge between $(1,1)$
and $(2,1)$ and the weight from the vertex $(1,1)$ as in the initial
configuration we are considering particle 1 has already jumped over
hole 1. We colour $Q(1,2)$ red, $Q(2,1)$ blue and every other block
$Q(i,j)$ in $(\mathbb{R}_+)^2 \backslash \left[0,1\right]^2$ either
red if
$$ \widetilde{R}^{(1)}((1,2),(i,j)) > \widetilde{R}^{(1)}((2,1),(i,j)) $$
and blue if
$$ \widetilde{R}^{(1)}((1,2),(i,j)) \leq \widetilde{R}^{(1)}((2,1),(i,j)) $$
(recall the defintion of $R$ in (\ref{R}); $\widetilde{R}^{(1)}$ is the corresponding quantity in model R1 with the changes mentioned above). This implies that if $Q(i,j+1)$ is red and $Q(i+1,j)$ is blue, then $Q(i+1,j+1)$ is red iff
\begin{equation}
\widetilde{T}^{(1)}(i,j+1) \geq \widetilde{T}^{(1)}(i+1,j) \label{T1in1}
\end{equation}
and blue iff
\begin{equation}
\widetilde{T}^{(1)}(i,j+1) < \widetilde{T}^{(1)}(i+1,j) \label{T1in2}
\end{equation}
where $\widetilde{T}^{(1)}$ is defined as in (\ref{T1}) but now in the model R1 with the modifications described above. The line $\varphi_n^{(1)}$ separating the two clusters is again called the competition
interface and due to the way we defined the red and blue cluster we
have again that the competition interface corresponds to the path of the second class particle. We can rewrite (\ref{T1in1}) and (\ref{T1in2}) in terms of $\varphi_n^{(1)}$ as
\begin{equation} \label{phi1}
\varphi_{n+1}^{(1)}=
\begin{cases}
 \varphi_n^{(1)} + (1,0) & \widetilde{T}^{(1)}(\varphi_n^{(1)} + (1,0)) \leq \widetilde{T}^{(1)}(\varphi_n^{(1)} + (0,1)) \\
 \varphi_n^{(1)} + (0,1) & \widetilde{T}^{(1)}(\varphi_n^{(1)} + (1,0)) > \widetilde{T}^{(1)}(\varphi_n^{(1)} + (0,1))
\end{cases}
\end{equation}
Note the similarity of (\ref{phi}) with (\ref{phi1}); the difference comes from the fact that in the models
in discrete time ties are possible. As in \cite{ferraripimentel} we
want to prove that this competition interface has almost surely an
asymptotic direction. First we note that the results in
\cite{ferraripimentel} about geodesics still hold with geometric weights attached to
the vertices instead of exponential weights. Secondly, the results
still hold in a model where `$+1$' weights are attached to
\textit{every} horizontal edge in $\ZZ_+^2$ as these weights do not
affect the geodesics (they just give a constant weight to every path
from $z$ to $z'$, $z,z' \in \ZZ_+^2$). The only difference in our case
is that there is no weight attached to the edge from $(1,1)$ to
$(2,1)$. But this local change does not affect the almost sure
statements in Propositions 7,8 and 9 in \cite{ferraripimentel}. We
conclude that the competition interface in our model has almost sure
an asymptotic direction and it follows from arguments analogous to the
ones in \cite{ferraripimentel} that the speed of the second class
particle converges almost surely.

\subsubsection*{R2}
The result for model R2 follows again from the symmetry between R1 and R2.

\subsubsection*{R3}
For the purpose of this section it is convenient to change the
last-passage percolation model corresponding to model R3 a little
bit. Instead of updating the even sites and then the odd sites during 
a single time-step, we separate the two batches of 
updates by a half time-step. Then
the percolation model corresponding to the initial configuration
$$\eta_0(x) =
\begin{cases}
	1 & x \leq 0 \\
	0 & x \geq 1
\end{cases}
$$
has no weights attached to the edges, while the weights at the vertices are 
geometric with an extra $\frac{1}{2}$ added. As noticed in the beginning of
this section, introducing an extra site into this model changes the
parity of some sites. In order to deal with this we attach a single
`$+\frac{1}{2}$' weight to the edge from $(1,1)$ to $(1,2)$ in the
percolation model corresponding to the initial configuration
$$\eta_0(x) =
\begin{cases}
	1 & x \leq -1 \\
	0 & x = 0 \\
	1 & x = 1 \\
	0 & x \geq 2
\end{cases}
$$
(where we also remove the weight from the vertex $(1,1)$). This ensures that we apply
even/odd updates to the left of the particle of the pair and odd/even
updates to the right. Then the movement of the pair $\varphi_n^{(3)}$ corresponds to
the movement of the second class particle in a model with even/odd updates.

Again we colour $Q(1,2)$ red, $Q(2,1)$ blue and now every other block $Q(i,j)$ in $(\mathbb{R}_+)^2 \backslash \left[0,1\right]^2$ either red if
\begin{equation} \label{R31}
\widetilde{R}^{(3)}((1,2),(i,j)) \geq \widetilde{R}^{(3)}((2,1),(i,j))
\end{equation}
and blue if
\begin{equation} \label{R32}
\widetilde{R}^{(3)}((1,2),(i,j)) < \widetilde{R}^{(3)}((2,1),(i,j))
\end{equation}
The interface between these clusters is again the competition interface and is given by $\varphi_n^{(3)}$. In terms of $\varphi_n^{(3)}$ we get from (\ref{R31}) and (\ref{R32}) that
\begin{equation} \label{phi3}
\varphi_{n+1}^{(3)}=
\begin{cases}
 \varphi_n^{(3)} + (1,0) & \widetilde{T}^{(3)}(\varphi_n^{(3)} + (1,0)) < \widetilde{T}^{(3)}(\varphi_n^{(3)} + (0,1)) \\
 \varphi_n^{(3)} + (0,1) & \widetilde{T}^{(3)}(\varphi_n^{(3)} + (1,0)) > \widetilde{T}^{(3)}(\varphi_n^{(3)} + (0,1))
\end{cases}
\end{equation}
$\widetilde{R}^{(3)}$ and $\widetilde{T}^{(3)}$ are defined as in (\ref{R}) and (\ref{T3}) but now we are considering the model R3 with the modifications described above. Due to the additional `$+\frac{1}{2}$' weight on the edge from $(1,1)$ to $(1,2)$ the competition interface never encounters any ties in this model, i.e. $\widetilde{T}^{(3)}(\varphi_n^{(3)} + (1,0)) = \widetilde{T}^{(3)}(\varphi_n^{(3)} + (0,1))$ does not occur. As in R1, the local change given by the extra `$+\frac{1}{2}$' edge-weight in
this model does not change the almost sure statements in Propositions
7,8 and 9 in \cite{ferraripimentel}. The rest of the argument is the
same as for model R1.


\end{proof}

\begin{remark} \label{proofdirR1R3}
We can use the known distributions of the speeds of the second class particles
(see Theorem \ref{2CP}) together with the hydrodynamic limit results (see Theorem \ref{hydro}
and Remark \ref{shape}) to prove the interesting result mentioned in Remark \ref{dirR1R3}
that the distributions of the asymptotic direction of the competition interfaces in the models
R1 and R3 are the same:
\begin{proof}
The proof exploits the connections made in Proposition 5 in \cite{ferraripimentel}. Similar to
\cite{ferraripimentel} we let $\psi^{(1)}_t = (I^{(1)}(t),$
$J^{(1)}(t))$, $\psi^{(3)}_t = (I^{(3)}(t),J^{(3)}(t))$ be the
position of the competition interface (i.e. the labels of the pair) at
time $t$ and denote by $\theta^{(1)}, \theta^{(3)} \in
\left[0,90\text{\textdegree}\right]$ the random angle of the
competition interface for model R1 and R3, i.e.
$$ \lim_{t \rightarrow \infty} \frac{\psi^{(i)}_t}{\left| \psi^{(i)}_t \right|} = e^{i\theta^{(i)}} = \left(\cos(\theta^{(i)}), \sin(\theta^{(i)})\right) \text{ for } i=1,3 $$
By the arguments in the previous sections we know that these limits exist almost surely. Using the asymptotic shape of the growth models given in Remark \ref{shape} we also have
$$ \lim_{t \rightarrow \infty} \frac{\psi^{(i)}_t}{t} = j^{(i)}(\theta^{(i)})e^{i\theta^{(i)}} \text{ a.s. for } i=1,3 $$
where $j^{(i)}(\theta^{(i)})$ is the distance from the origin to the intersection of the line given by $\{ (u,v) \in \mathbb{R}_+^2 : \tan(\theta^{(i)}) = \frac{v}{u} \}$ and the asymptotic \textit{growth interface} $(x,g_i(x))$ for $i=1,3$. With the formulas for $g_1$ and $g_3$ in (\ref{g1}) and (\ref{g3}) we can calculate $j^{(1)}(\theta^{(1)})$ and $j^{(3)}(\theta^{(3)})$ explicitly:
$$ j^{(1)}(\theta^{(1)}) = \frac{\beta}{\left( \sqrt{ \left( 1-\beta \right) \sin(\theta^{(1)})} + \sqrt{\cos(\theta^{(1)})} \right)^2} $$
and
$$ j^{(3)}(\theta^{(3)}) = \frac{2\beta \left( 2-\beta \right)}{\left( \left( 2-\beta \right) \sqrt{\sin(\theta^{(3)})}
+ 2\sqrt{\left( 1-\beta \right) \cos(\theta^{(3)})} \right)^2 + \beta^2 \cos(\theta^{(3)})} $$
By the connection between the path of the second class particle and the competition interface ($X^{(i)}(t) = I^{(i)}(t) - J^{(i)}(t)$, $i=1,3$, since the second class particle moves to the right iff the first label of the pair increases and to the left iff the second label of the pair increases) it follows that
$$ \lim_{t \rightarrow \infty} \frac{X^{(i)}(t)}{t} = \lim_{t \rightarrow \infty} \frac{I^{(i)}(t) - J^{(i)}(t)}{t} = j^{(i)}(\theta^{(i)}) \left( \cos(\theta^{(i)}) - \sin(\theta^{(i)}) \right) \stackrel{\mathrm{def}}= l_i(\theta^{(i)}) $$
almost surely for $i=1,3$. Using the known distributions of the speed of the second class particle in model R1 and R3 (see Theorem \ref{2CP}) we can calculate
$$ \mathbb{P}\left[ \theta^{(1)} \leq \alpha \right] = \mathbb{P}\left[ l_1(\theta^{(1)}) \geq l_1(\alpha) \right] = f_1(l_1(\alpha)) $$
and
$$ \mathbb{P}\left[ \theta^{(3)} \leq \alpha \right] = \mathbb{P}\left[ l_3(\theta^{(3)}) \geq l_3(\alpha) \right] = a_3(l_3(\alpha)) $$
A calculation shows that
$$ f_1(l_1(\alpha)) = a_3(l_3(\alpha)) $$
\end{proof}
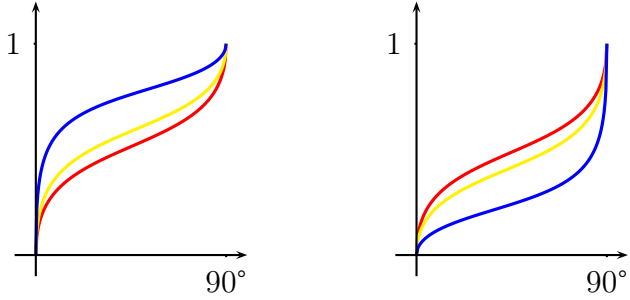
\begin{figure}[t]
\begin{pspicture}(-3,-1)(1,3)
\psset{xunit=0.8pt, yunit=80pt}
\psplot[linewidth=1.2pt, plotpoints=200, linecolor=red]{0}{90}{x cos 0.1 div sqrt x sin 0.9 mul 0.1 div sqrt add 2 exp -1 exp x cos x sin sub mul -1 mul 1 add -1 exp 0.9 mul sqrt -1 mul 1 add 0.1 div}
\psplot[linewidth=1.2pt, plotpoints=200, linecolor=yellow]{0}{90}{x cos 0.5 div sqrt x sin 0.5 mul 0.5 div sqrt add 2 exp -1 exp x cos x sin sub mul -1 mul 1 add -1 exp 0.5 mul sqrt -1 mul 1 add 0.5 div}
\psplot[linewidth=1.2pt, plotpoints=200, linecolor=blue]{0}{90}{x cos 0.9 div sqrt x sin 0.1 mul 0.9 div sqrt add 2 exp -1 exp x cos x sin sub mul -1 mul 1 add -1 exp 0.1 mul sqrt -1 mul 1 add 0.9 div}
\psline{->}(0,-0.1)(0,1.2)
\psline{->}(-10,0)(100,0)
\psline(90,0)(90,-0.01)
\psline(-1,1)(0,1)
\uput{0.2}[270](90,0){90\textdegree}
\uput{0.2}[180](0,1){1}
\psplot[linewidth=1.2pt, plotpoints=200, linecolor=red]{180}{270}{x 180 sub sin 0.1 div sqrt x 180 sub cos 0.9 mul 0.1 div sqrt add 2 exp -1 exp x 180 sub cos x 180 sub sin sub mul 1 add -1 exp 0.9 mul sqrt -1 mul 1 add 0.1 div -1 mul 1 add}
\psplot[linewidth=1.2pt, plotpoints=200, linecolor=yellow]{180}{270}{x 180 sub sin 0.5 div sqrt x 180 sub cos 0.5 mul 0.5 div sqrt add 2 exp -1 exp x 180 sub cos x 180 sub sin sub mul 1 add -1 exp 0.5 mul sqrt -1 mul 1 add 0.5 div -1 mul 1 add}
\psplot[linewidth=1.2pt, plotpoints=200, linecolor=blue]{180}{270}{x 180 sub sin 0.9 div sqrt x 180 sub cos 0.1 mul 0.9 div sqrt add 2 exp -1 exp x 180 sub cos x 180 sub sin sub mul 1 add -1 exp 0.1 mul sqrt -1 mul 1 add 0.9 div -1 mul 1 add}
\psline{->}(180,-0.1)(180,1.2)
\psline{->}(170,0)(280,0)
\psline(270,0)(270,-0.01)
\psline(179,1)(180,1)
\uput{0.2}[270](270,0){90\textdegree}
\uput{0.2}[180](180,1){1}
\end{pspicture}
\setcaptionmargin{1cm}
\caption{Distribution function for the random angle of the competition interface in the models R1 and R3 (left) and R2 (right) for $\beta=0.1$ (red), $\beta=0.5$ (yellow) and $\beta=0.9$ (blue)}
\end{figure}
\end{remark}

\begin{proof}[Proof of Theorem \ref{duality2}:]
Let the operators $\sigma_n$ be defined by
$$ \sigma_n Y =
\begin{cases}
	\tau_n Y & Y_n < Y_{n+1} \\
	Y & \text{otherwise}
\end{cases}
$$
where $\tau_n$ exchanges $Y_n$ and $Y_{n+1}$ in $Y$. The proof of the following general Lemma is the same as in \cite{amirangelvalko} (Lemma 3.1).
\begin{lemma}\label{sigma}
For a fixed sequence $i_1,\ldots,i_k$ in $\mathbb{Z}$ we have
$$ \sigma_{i_k} \cdots \sigma_{i_1} \overset{d}{=} \left( \sigma_{i_1} \cdots \sigma_{i_k}  \right)^{-1} $$
\end{lemma}
We have the relation $Y_{X_n(m)}(m) = n$ between the $X$ and the $Y$ process. Since $\beta < 1$ each site has a positive probability that no jump occurs at that site at any given time. At each time-step these sites separate $\mathbb{Z}$ into finite intervals and the events on these intervals during that time-step are independent. In the model with updates from right to left we apply a finite sequence of operators $\sigma_{i_1} \cdots \sigma_{i_k}$ where $i_1,\ldots,i_k$ is an increasing sequence (since we update from right to left). Lemma \ref{sigma} states that applying $\sigma_{i_1} \cdots \sigma_{i_k}$ is the same (in distribution) as applying $\sigma_{i_k} \cdots \sigma_{i_1}$ (i.e. updating from left to right) and taking the inverse permutation. But given the configuration $Y_n(0)$ and performing the updates from left to right we get $Y_{n}^{(2)}(1)$ and this is the inverse permutation of $X_{n}^{(2)}(1)$. So
$$ Y_{n}^{(1)}(1) \overset{d}{=} X_{n}^{(2)}(1)$$
Inductively we get that this holds for all $m>0$. The other three parts follow in the same way. In the case with even/odd updates, $i_1,\ldots,i_k$ is a sequence such that there exists a $1 \leq j \leq k+1$ such that $i_l$ is odd for $l < j$ and $i_l$ is even for $l \geq j$.
\end{proof}

In order to prove Theorem \ref{duality} we will need the following Lemma. We state it here for the model R1, but analogous results hold for the other models as well. The Lemma corresponds to Lemma 4.1 in \cite{amirangelvalko}.
\begin{lemma}\label{speedshift}
Consider two TASEPs, $Y^{(1)}$ and $\widetilde{Y}^{(1)}$, as functions of the same Bernoulli points on $\mathbb{Z} \times \mathbb{N}$ (i.e. under basic coupling). We set $Y_{n}^{(1)}(0) = n$ and $\widetilde{Y}_{n}^{(1)}(0) = \sigma_j \cdots \sigma_{j+k} Y_{n}^{(1)}(0)$ for some $j \in \mathbb{Z}$ and $k \geq 0$. Let $\{ U_{n}^{(1)} \}$ be the \textit{speed process} of $Y^{(1)}$ and $\{ \widetilde{U}_{n}^{(1)} \}$ be the \textit{speed process} of $\widetilde{Y}^{(1)}$. Then $\widetilde{U}^{(1)} = \sigma_{j+k} \cdots \sigma_j U^{(1)}$.
\end{lemma}
\begin{proof}
Every other particle than $\{ j, \ldots j+k+1 \}$ is either stronger than all particles $\{ j, \ldots , j+k+1 \}$ or weaker than all particles $\{ j, \ldots , j+k+1 \}$. Any swap of a particle other than $\{ j, \ldots , j+k+1 \}$ will happen in both $Y^{(1)}$ and $\widetilde{Y}^{(1)}$. So for any $i \notin \{ j, \ldots , j+k+1 \}$ we have $X_{i}^{(1)}(m) = \widetilde{X}_{i}^{(1)}(m)$ for all $m \geq 0$ and therefore $U_{i}^{(1)} = \widetilde{U}_{i}^{(1)}$ for those $i$. In $\widetilde{Y}^{(1)}$ particle $j+k+1$ is always to the left of all other particles $\left\{ j, \ldots , j+k \right\}$. So $\widetilde{U}_{j+k+1}^{(1)} = \min \{ U_{j}^{(1)} \ldots U_{j+k+1}^{(1)} \}$. Define $j \leq r \leq j+k+1$ by
$$ \min \left\{ U_{j}^{(1)} \ldots U_{j+k+1}^{(1)} \right\} = U_{r}^{(1)} $$
Then for $i = r, r+1, \ldots , j+k$ we have $\widetilde{U}_{i}^{(1)} = U_{i+1}^{(1)}$ and for $i=j,j+1, \ldots , r-1$
\begin{align*}
&\widetilde{U}_{j}^{(1)} = \max \left\{ U_{j}^{(1)} , U_{j+1}^{(1)} \right\} \\
&\widetilde{U}_{j+1}^{(1)} = \max \left\{ \min \left\{ U_{j}^{(1)} , U_{j+1}^{(1)} \right\} , U_{j+2}^{(1)} \right\} \\
&\widetilde{U}_{j+2}^{(1)} = \max \left\{ \min \left\{ \min \left\{ U_{j}^{(1)} , U_{j+1}^{(1)} \right\} , U_{j+2}^{(1)} \right\} , U_{j+3}^{(1)} \right\} \\
&\ldots
\end{align*}
This shows that $\widetilde{U}^{(1)} = \sigma_{j+k} \cdots \sigma_j U^{(1)}$.
\end{proof}

\begin{proof}[Proof of Theorem \ref{duality}:]
Consider a Bernoulli process on $\mathbb{Z} \times \mathbb{Z}$. Half of this process ($\mathbb{Z} \times \mathbb{N}$) is used to construct the TASEP $Y^{(1)}$. For any $l \in \mathbb{Z}$ we can translate the Bernoulli process by $l$ (i.e. take points of the form $(n,m+l)$ where $(n,m)$ is in the original process). We can restrict this translated process to $\mathbb{Z} \times \mathbb{N}$ and use this restricted process to construct another TASEP. Let $U^{(1)}(l) = \{ U_{n}^{(1)}(l) \}$ be the speed process for the TASEP that has been constructed using the Bernoulli process translated by $l$. For every $l$, $U^{(1)}(l)$ has distribution $\mu^{(1)}$. So we have to show that $\{ U_{n}^{(1)}(l) \}$ behaves like a TASEP with updates from left to right. In order to do this we look at a transition $\{ U_{n}^{(1)}(l) \} \rightarrow \{ U_{n}^{(1)}(l+1) \}$. The effect on the original TASEP of changing from translating by $l$ to translating by $l+1$ is that some finite sequences of $\sigma$ operators of the form $\sigma_j \cdots \sigma_{j+k}$ are added to be applied to the TASEP before the original sequence of operations. At each location a $\sigma$ operator is added with probability $\beta$. The previous Theorem shows, that applying each of these finite sequences has the same effect on the speeds as applying each sequence in reverse order to the speed process. This shows that $\{ U_{n}^{(1)}(l) \}$ behaves like a TASEP with updates from left to right and therefore the measure $\mu^{(1)}$ is stationary for the TASEP with updates from left to right. \\
The proofs for the other three models are essentially the same (using the appropriate versions of Lemma \ref{speedshift}).
\end{proof}

The following Lemma will allow us to do the explicit calculations for the joint densities of the speeds in Theorems \ref{jointR1} - \ref{jointR3c} using the connection between queueing models and the invariant measures introduced in Theorem \ref{invariantmulti}. Here $D(g_i)$ is the domain of the distribution function $g_i$ ($i=0,1,2,3,4$).

\begin{lemma} \label{queue}
If $F: D(g_i) \rightarrow \{ 1 , \ldots , N \}$ is non-decreasing then for the TASEP $\{ Y_{n}^{(i)}(m) \}_{n,m}$ the distribution of $\{ F(U_{n}^{(i)}) \}_{n}$ is the unique ergodic stationary measure of the multi-type TASEP model R$j(i)$ with types $\{ 1 , \ldots , N \}$ and densities $\lambda_l = g_i ( \sup \{ F^{-1} ( l ) \} ) - g_i ( \inf \{ F^{-1} ( l ) \} )$ for type $l=1,\ldots,N$ ($j$ as in Theorem \ref{duality}).
\end{lemma}

\begin{proof}
The proof is analogous to the proof of Corollary 5.4 in \cite{amirangelvalko}.
\end{proof}

With the help of this Lemma we can do all the calculations needed for the results in Theorems \ref{jointR1} - \ref{jointR3c}. Depending on the model we are considering we will choose the function $F$ from Lemma \ref{queue} to be $F_i = \min \{ j : g_i(u) < x_j \}$ for some increasing sequence $(x_1,\ldots,x_{N-1})$ in $[0,1]$. Then we put $V_n = F(U_n)$ and the distribution of the $V$s is given by the invariant measure for the multi-type models and can be calculated explicitly using the queueing representation.

\begin{proof}[Proof of Theorem \ref{jointR1}:]
By Lemma \ref{queue} we have with $N=3$ that $F_1(U_n)$ is distributed according to the unique ergodic stationary measure of a 3-type TASEP with updates from left to right and densities
\begin{align*}
&\lambda_1 = \mathbb{P} \left[ x_0 < g_1(U_n) < x_1 \right] = x_1 - x_0 = x_1 \\
&\lambda_2 = \mathbb{P} \left[ x_1 < g_1(U_n) < x_2 \right] = x_2 - x_1 \\
&\lambda_3 = \mathbb{P} \left[ x_2 < g_1(U_n) < x_3 \right] = x_3 - x_2 = 1 - x_2
\end{align*}
$\lambda_1$ is the density of first class particles, $\lambda_2$ is the density of second class particles and $\lambda_3$ is the density of third class particles (or holes). Recall that $V_n = F_1(U_n)$. Using the queueing representation for the unique ergodic stationary measure of a 3-type TASEP we can calculate the joint distribution $\left( V_0, V_1 \right)$ explicitly. This distribution depends on the $x_i$. Taking suitable derivatives with respect to these $x_i$ we get the density of the corresponding speeds. We have for example
\begin{align*}
\mathbb{P} \left[ U_0 < g_1^{-1}(x_1) < U_1 < g_1^{-1}(x_2) \right]
&= \mathbb{P} \left[ V_0 = 1, V_1 = 2 \right] \\
&= x_1 x_2 \left( x_2 - x_1 \right)
\end{align*}
since the probability of having a second class particle at position 1 is $x_2 - x_1$ (since this is the density of second class particles) and to have a first class particle at position 0 we then have to have an arrival (probability $x_1$) and a service (probability $x_2$) because having a second class particle at site 1 means that the queue was empty at that time (so in order to have a departure at site 0 we need an arrival at site 0). Remember that in Theorem \ref{invariantmulti} the particles in the TASEP jumped from the right to the left. If we want to consider the TASEP with jumps from the left to the right (and that is what we are doing here) we have to read the queues from right to left. So position 1 comes before position 0 and the probability of having a second class particle at position 1 is independent of arrivals and services at position 0. So for $u_0 < u_1$ we put $x_1 = g_1(u_0)$ and $x_2 = g_1(u_1)$ and get as density
\begin{align*}
\mathbb{P} \left[ U_0 \in du_0 , U_1 \in du_1 \right]
&= \frac{dx_1}{du_0} \frac{dx_2}{du_1} \frac{d}{dx_1} \frac{d}{dx_2} x_1 x_2 \left( x_2 - x_1 \right) \\
&= \frac{1-\beta}{4\beta^2} \left( 1 - u_0 \right)^{-\frac{3}{2}} \left( 1 - u_1 \right)^{-\frac{3}{2}} \left( 2 g_1(u_1) - 2 g_1(u_0) \right) \\
& = \frac{1-\beta}{2\beta^3} \left( 1 - u_0 \right)^{-\frac{3}{2}} \left( 1 - u_1 \right)^{-\frac{3}{2}} \\
&\qquad \qquad \qquad \qquad \cdot \left( \sqrt{\frac{1-\beta}{1-u_1}} - \sqrt{\frac{1-\beta}{1-u_0}} \right)
\end{align*}
Similarly, we have
\begin{align*}
\mathbb{P} \left[ g_1^{-1}(x_1) < U_1 < g_1^{-1}(x_2) < U_0 \right]
&= \mathbb{P} \left[ V_0 = 3, V_1 = 2 \right] \\
&= \left( 1 - x_2 \right) \left( x_2 - x_1 \right)
\end{align*}
and therefore we get as density for $u_0 > u_1$ (putting $x_1 = g_1(u_1)$, $x_2 = g_1(u_0)$)
\begin{align*}
\mathbb{P} \left[ U_0 \in du_0 , U_1 \in du_1 \right]
&= \left( - \frac{dx_1}{du_0} \right) \left( - \frac{dx_2}{du_1} \right) \frac{d}{dx_1} \frac{d}{dx_2} \left( 1 - x_2 \right) \left( x_2 - x_1 \right) \\
&= \frac{1-\beta}{4\beta^2} \left( 1 - u_0 \right)^{-\frac{3}{2}} \left( 1 - u_1 \right)^{-\frac{3}{2}} \\
&= g_1^{'}(u_0)g_1^{'}(u_1)
\end{align*}
To get the density for $u_0 = u_1$ we consider
\begin{align*}
\mathbb{P} \left[ g_1^{-1}(x_1) < U_0, U_1 < g_1^{-1}(x_2) \right]
&= \mathbb{P} \left[ V_0 = 2, V_1 = 2 \right] \\
&= \left( 1 - x_1 \right) x_2 \left( x_2 - x_1 \right)
\end{align*}
and let $x_1, x_2 \rightarrow g_1(u)$. We get
\begin{align*}
\mathbb{P} \left[ U_0 , U_1 \in du \right]
&= \lim_{x_1,x_2 \rightarrow g_1(u)} \frac{\left( 1 - x_1 \right) x_2 \left( x_2 - x_1 \right)}{g_1^{-1}(x_2) - g_1^{-1}(x_1)} \\
&= \frac{\sqrt{1-\beta}\left( 1 - g_1(u) \right) g_1(u)}{2 \beta \left( 1 - u \right)^{\frac{3}{2}}} \\
&= \frac{\sqrt{1-\beta}}{2 \beta^2 \left( 1 - u \right)^{\frac{3}{2}}} \left( 1 - \frac{1}{\beta} \right) + \frac{1-\beta}{2 \beta^2 \left( 1 - u \right)^{2}} \left( \frac{2}{\beta} - 1 \right) \\
&\qquad \qquad - \frac{\sqrt{1-\beta}\left( 1 - \beta \right)}{2 \beta^3 \left( 1 - u \right)^{\frac{5}{2}}}
\end{align*}
To get the probabilities in the Theorem we only have to integrate the densities over the appropriate ranges of $u_0$, $u_1$ and $u$. Alternatively we can use the following:
\begin{align*}
\mathbb{P} \left[ U_0^{(1)} < U_1^{(1)} \right]
&= \mathbb{P} \left[ g_1(U_0^{(1)}) < g_1(U_1^{(1)}) \right] \\
&\overset{(*)}{=} \mathbb{P} \left[ g_0(U_0^{(0)}) < g_0(U_1^{(0)}) \right] \\
&= \mathbb{P} \left[ U_0^{(0)} < U_1^{(0)} \right] \\
&= \frac{1}{3}
\end{align*}
$(*)$ follows from the fact that the distribution of $\{ g_1(U_n^{(1)})\}$ is the unique translation invariant stationary ergodic measure for the TASEP R2 with marginals uniform on $[0,1]$. The distribution of $\{ g_0(U_n^{(0)})\}$ is the unique translation invariant stationary ergodic measure for the TASEP in continuous time. Since the stationary distributions for the multi-type TASEPs R0 and R2 are the same (see Theorem \ref{invariantmulti}) $\{ g_1(U_n^{(1)})\}$ has the same distribution as $\{ g_0(U_n^{(0)}) \}$.
\end{proof}

The proofs for Theorems \ref{jointR3a} - \ref{jointR3c} work in exactly the same way.

\section{Fully parallel updates}
Finally we mention the model with ``fully parallel updates''. 
If an update occurs at site $x$ at time $t$
(which happens with probability $\beta$ as usual),
this update causes a jump from $x$ to $x+1$ if and only if 
$\eta_{t-1}(x)=1$ and $\eta_{t-1}(x+1)=0$ 
(that is, the jump is already possible before 
any other updates at the current time-step are performed).

There are several important differences between this model and the 
models we have studied earlier.
The Bernoulli product measures
$\nu_{\rho}$ are no longer invariant. Furthermore, the 
basic coupling no longer preserves an ordering between
different initial configurations, and so it is no longer
clear how to define a multi-class system. 
If we use basic coupling to couple two
systems which start with one discrepancy at the origin then this single
discrepancy can generate additional discrepancies. It would already be
interesting to know how the leftmost and rightmost discrepancies
behave. Do they have asymptotic speeds, and if so are the speeds random or
deterministic? There is still a natural percolation representation,
and we can still obtain a hydrodynamic limit result in the
sense that $\frac{1}{n} \sum_{un < k < vn} \eta_t(n)$ converges
a.s. to the constant value $\int_{u}^{v} f(w) dw$, for $u < v$ and
some function $f$, but the stronger result that $\lim_{n \rightarrow
  \infty} \mathbb{E} \left[ \eta_n(k) \right]$ exists and is equal to
$f(u)$ whenever $\frac{k}{n}$ tends to $u$ does not follow using the
same methods as in the proof of Theorem \ref{hydro}.

\section*{Acknowledgments}
JM was supported by the EPSRC. PS was supported by a ``DAAD Doktorandenstipendium'' and the EPSRC.

\bibliographystyle{plain}
\bibliography{article}

\bigskip

\parbox{0.33\textwidth}{\noindent
 James Martin\\
 Department of Statistics\\
 University of Oxford\\
 1 South Parks Road  \\
 Oxford OX1 3TG\\
 United Kingdom\\
\texttt{martin@stats.ox.ac.uk}}
\hfill
\parbox{0.33\textwidth}{\noindent
 Philipp Schmidt \\
 Department of Statistics\\
 University of Oxford\\
 1 South Parks Road  \\
 Oxford OX1 3TG\\
 United Kingdom\\
\texttt{schmidt@stats.ox.ac.uk}}

\end{document}